\documentclass[letterpaper]{article} 
\usepackage{uai2020}
\usepackage[margin=1in]{geometry}
\usepackage{times}

\usepackage{hyperref}
\usepackage{url}
\usepackage{amsmath}

\usepackage{microtype}
\usepackage{graphicx}
\usepackage{subfigure}
\usepackage{booktabs} 
\usepackage{amsmath}
\usepackage{amsthm}
\usepackage{amssymb}
\usepackage{xcolor}
\usepackage{textcomp}
\usepackage{wrapfig}
\usepackage{framed}
\usepackage{color}
\usepackage{enumerate}
\usepackage{enumitem}
\usepackage{graphicx}
\usepackage{natbib}
\usepackage{algorithm}
\usepackage{algpseudocode}

\newtheorem{thm}{Theorem}
\newtheorem*{thm*}{Theorem}

\newtheorem{asu}[thm]{Assumption}
\newtheorem{cor}[thm]{Corollary}
\newtheorem{prop}[thm]{Proposition}
\newtheorem*{prop*}{Proposition}
\newtheorem*{cor*}{Corollary}

\newtheorem{lem}[thm]{Lemma}

\newcommand{\reals}{\mathbb{R}}

\newcommand{\mc}[1]{\mathcal{#1}}

\newcommand{\Mod}[1]{\ (\mathrm{mod}\ #1)}
\newcommand{\fw}{\texttt}
\theoremstyle{definition}
\newtheorem{example}[thm]{Example}
\newtheorem{definition}[thm]{Definition}

\title{Bounding the expected run-time of nonconvex optimization with early stopping}

\author{  {\bf Thomas Flynn}
  \And {\bf Kwang Min Yu} \And {\bf Abid Malik}  \\ \\
   Computational Science Initiative \\
  Brookhaven National Laboratory \\
  Upton, NY 11973  \\\And {\bf Nicolas D'Imperio} \And {\bf Shinjae Yoo}  
 }
\begin{document}
\maketitle

\begin{abstract}
  This work examines the convergence of stochastic gradient-based optimization algorithms that use early stopping based on a validation function. The form of early stopping we consider is that optimization terminates when the norm of the gradient  of a validation function falls below a threshold. We derive conditions that guarantee this stopping rule is well-defined, and provide bounds on the expected number of iterations and gradient evaluations needed to meet this criterion. The guarantee accounts for the distance between the training and validation sets, measured with the Wasserstein distance. We develop the approach in the general setting of a first-order optimization algorithm, with possibly biased update directions subject to a geometric drift condition. We then derive bounds on the expected running time for early stopping variants of several algorithms, including stochastic gradient descent (SGD), decentralized SGD (DSGD), and the stochastic variance reduced gradient (SVRG) algorithm. Finally, we consider the generalization properties of the iterate returned by early stopping.
\end{abstract}

\section{\uppercase{Introduction}}
This work considers the minimization of a differentiable  and possibly nonconvex objective function:
  \begin{equation}\label{main-opt-problem}
    \min_{x\in \reals^d}\, f(x).
  \end{equation}
For nonconvex problems, a  generally accepted notion of success for algorithms that use only first-order information is that an \textit{approximate stationary point} is generated. These are points $x \in \reals^d$ where the norm of the gradient of $f$ is small.    In a typical machine learning scenario, $f$ is the average loss over a dataset of training examples, and it is common to solve  problem \eqref{main-opt-problem}
  using stochastic gradient-based optimization, for instance, stochastic gradient descent  (SGD; see Algorithm \ref{algo1}). The success of SGD in machine learning problems has led to many extensions of the algorithm, including variance-reduced and distributed variants (reviewed in Section 1.1).

  A common approach to stopping optimization in practice is to use early stopping, in which a performance criterion is periodically evaluated on a validation function and optimization terminates once this condition is met. 
  However, there is little theoretical work on the run-time of nonconvex optimization with such early stopping rules. In general, one would expect that the run-time and performance will depend on several factors, including the similarity between the validation and training functions, the desired level of solution accuracy, and internal settings of the optimization algorithm, such as  learning rates.
  
  In this work, we carry out an analysis of early stopping when the criterion is that the algorithm has generated an approximate stationary point for a validation function.
  Formally, we consider the stopping time defined as the first time, or iteration number, that an iterate has the property of being an approximate stationary point for the validation function, and we derive upper bounds on the expected value of this stopping time. Furthermore, although the stopping time is defined in terms of stationarity of the \textit{validation} function, we also derive a bound on the stationarity gap of the \textit{training} function at the resulting iterate, in terms of the Wasserstein distance between the training and validation sets. As an extension, we describe how to leverage Wasserstein concentration results to bound the expected stationarity gap with respect to the \textit{testing} distribution from which both the training and validation datasets are drawn.

The analysis is carried out for several procedures, including stochastic gradient descent (SGD),  decentralized SGD (DSGD), and the stochastic variance reduced gradient (SVRG) algorithm,  The result is new bounds on the expected number of Incremental First-order Oracle (IFO) calls needed by these algorithms to generate approximate stationary points. We believe the general technique used to obtain the results will be useful for analyzing the expected running time of other optimization algorithms as well.

\,
\noindent\textbf{Main contributions.} Our main contributions include:
\begin{itemize}[leftmargin=1em]
\item[$\circ$] We present a non-asymptotic analysis of SGD with early stopping that leads to a bound on the expected amount of resources needed to find approximate stationary points of the training function, including the number of iterations (Proposition \ref{prop:abstract-bias}) and  gradient evaluations (Corollary \ref{sgd-ifo}). The analysis allows for biases in the update direction, subject to a geometric drift condition on the error terms (specified in Assumption \ref{asu:bias-grad}.)
\item[$\circ$] We specialize the results  to decentralized SGD, a variant of SGD designed for distributed computation, resulting in upper bounds on  the number of iterations  (Proposition  \ref{decentralized-main}) and gradient evaluations (Corollary  \ref{cor:decentralized-ifo}) needed by the algorithm. This is done by modeling DSGD as a biased form of SGD, whose bias is controlled in part by the diffusion coefficient of the network communication graph.
\item[$\circ$] We derive a run-time bound for a variant of nonconvex SVRG with early stopping, obtaining a bound on the expected number of iterations (Proposition  \ref{prop:svrg-conv}) and  IFO calls (Corollary  \ref{svrg-ifo}) need to generate  approximate stationary points.
\item[$\circ$] We demonstrate how Wasserstein concentration bounds can be leveraged to bound the generalization performance of the iterate returned by  the algorithms (Section \ref{sect:genprop}), expressed in terms of the number of samples used to construct the datasets, and properties of the testing distribution.
\end{itemize}
\subsection{Related work}
The study of stochastic optimization goes back (at least) to the pioneering work of Robbins and Monro \citep{robbins1951}.  Subsequent developments include the ordinary differential equation (ODE) method \citep{ljung1977analysis} and stochastic approximation \citep{kushner1978stochastic}, which emphasizes the asymptotic behavior of the algorithms.
Asymptotic performance of biased SGD has been considered in \citep{bertsekas} which establishes the asymptotic convergence of the algorithm to stationary points.
 
The randomized stochastic gradient (RSG) method \citep{ghadimi-lan} uses randomization to obtain a non-asymptotic performance guarantee for SGD applied to nonconvex functions. In one interpretation of RSG, the algorithm (e.g., SGD) is run for a fixed number of iterations, and  a random iterate is selected as the final output of optimization (alternatively, the algorithm is executed for a random number of steps, after which the final iterate is returned.) The randomization technique has became a standard tool for analyzing optimization algorithms in the nonconvex setting
  \citep{ghadimi2016accelerated,
    decentralized,
    reddi2016stochastic,
    riemannian-svrg,
    scsgnonconvex,
    lian2015}.
  Follow-up works have employed randomization for the analysis of nonconvex optimization in diverse algorithmic settings, such as asynchronous \citep{lian2015} and decentralized \citep{decentralized} optimization.
  
 Machine learning problems often involve an objective that is a finite sum of functions, and, in this setting, variance reduction techniques lead to improved rates of convergence over SGD \citep{svrg,sag,saga}.  Analysis of variance reduction has extended beyond convex functions, from an application to principal components analysis \citep{Shamir2015} to general nonconvex functions  \citep{zeyuan-svrg,  reddi2016stochastic,scsgnonconvex}. 
  
  In contrast to the aforementioned works, in which randomization is used to analyze performance in the nonconvex case, this work considers algorithms that use a  different approach to stopping optimization, based on periodically evaluating a performance criterion with respect to a validation function. There are several variants of early stopping that appear in practice. For instance, one approach is to train until the error on a validation set increases \citep{wang12}, \citep{semi2015}, or there is no improvement for a number of epochs \citep{perceptual19}. Alternatively, one can train the model for a fixed number of epochs, and then take the parameter from the epoch at which validation error is lowest \citep{jaderberg17a}, \citep{lee18c}, \citep{franceschi19a}.
 Despite the prevalence of early stopping, there is comparatively little work on the analysis of this strategy in the nonconvex setting, and our work aims to fill this gap
 
 Several recent works also have explored the average amount of resources needed to reach a desired performance level in optimization. The expected running time of a stochastic trust region algorithm
 is given in \citep{blanchettime}, based on a renewal-reward martingale argument. This proof technique was also used to analyze the expected run time of a  stochastic line search method \citep{scheinberglinesearch}. 
  Our convergence analysis is similar in spirit, as we also are interested in the expected amount of time or other resources required to meet the performance guarantee. However, our focus is on different algorithms (SGD, DSGD, and SVRG), and the variants of these algorithms that we consider contain explicit stopping mechanisms based on validation functions.
Other recent work considering the theoretical aspects of early stopping include \citep{duvenaud16}, where the authors developed an interpretation of early stopping in terms of variational Bayesian inference. Early stopping for a least squares problem in a reproducing kernel Hilbert space has been treated in \citep{lin2016optimal}, while the implications of early stopping on  generalization were studied in \citep{tfgb}. To the authors' knowledge, the present work is the first to analyze run-time when using a validation function for early stopping in nonconvex optimization.

\section{\uppercase{Preliminaries}}
Let
$f : \reals^q \times \reals^d \to \reals$
be a loss function whose value we denote by $f(y,x)$. Intuitively, the variable $y$ represents an (input, output) pair, and $x$ represents the parameters of a model. Throughout, we shall assume that the gradient of the loss function  is Lipschitz continuous:
  \begin{asu}\label{asu:lipgrad}
    The function
    $f:\reals^q \times \reals^d \to \reals$ is
    bounded from below by
    $f^{*} \in \reals$, and $\nabla_x f$ is $L$-Lipschitz continuous as a function of $x$: $ \forall y \in \reals^q, x_1,x_2 \in \reals^d,$
    \[\|\nabla_x f(y,x_1) - \nabla_x f(y,x_2)\| \leq L \|x_1 - x_2\|. \]
    \end{asu}
This is a standard assumption that is also referred to as smoothness of the loss function.  Where appropriate, we will make a distinction between the training function $f_T$, which is used to calculate gradients, and a validation function $f_V$ used to decide when to stop training:
  \begin{asu}\label{asu:finite-sum-fns}
  The function $f_T$ is defined using a set $Y_T \subseteq \mathbb{R}^{q}$  as
  $
  f_{T}(x) = \frac{1}{n_{T}}\sum_{y \in Y_T}f(y,x)$, where $n_T = |Y_T|$,
     and the function $f_V$ is defined using a set 
  $Y_V \subseteq \mathbb{R}^{q}$ as
  $
  f_{V}(x) = \frac{1}{n_{V}}\sum_{y \in Y_V}f(y,x)$,
  where $n_V = |Y_V|$.
  \end{asu}
  Note that there is no assumption that the validation and training sets are disjoint.
  At times we will assume  a bound on the variance of stochastic gradients of $f_T$:
  \begin{asu}\label{asu:ft-var-bias}
    There is a  $\sigma_v^2 \geq 0$ such that
    $\forall\, \,x\in\reals^d$,
    $$
    \frac{1}{n_{T}}\sum\limits_{y \in Y_{T}}\left\| \nabla_x f(y,x) - \nabla f_{T}(x)\right\|^{2}
    \leq 
    \sigma_v^2.$$
    \end{asu}
  In the SGD and DSGD algorithms considered below, optimization takes place on the training function, while the stopping criteria is evaluated using the validation function. To guarantee that this leads to well-defined behavior, we will make use of a bound on the distance between the training and validation functions. Intuitively, the functions $f_T$ and $f_V$ will be close when the datasets $Y_T$ and $Y_V$ are similar. Formally, the datasets $Y_T$ and $Y_V$ determine probability measures $\mu_T$ and $\mu_Y$, defined as 
  $\mu_T = \frac{1}{n_T}\sum_{y\in Y_T}\delta_y$ and $\mu_V = \frac{1}{n_V}\sum_{y\in Y_V}\delta_y,$
  respectively, 
  where $\delta_y$ is the delta measure  $\delta_y(A) = 1_{y \in A}$ for all sets $A$.
 We can compare these measures using the Wasserstein distance, which is defined as follows.

 For $q\geq 1, p\geq 1$, denote by 
$\mathcal{P}_p(\reals^q)$ 
the probability measures on 
$\reals^q$
with finite moments of order $p$. 
Recall that a coupling of probability measures $\mu_1$ and $\mu_2$ is a probability measure $\gamma$ on $\reals^q\times \reals^q$ such that for all measurable sets $A$, 
$\gamma(A \times \reals^q) = \mu_1(A)$ and 
$\gamma(\reals^q \times A) = \mu_2(A)$.  Intuitively, a coupling transforms data distributed like $\mu_1$ into a data  distributed according to $\mu_2$ (and vice versa).
The $p$-Wasserstein distance on $\mc{P}_{p}(\reals^q)$, denoted by $d_p$, is defined as:
\begin{equation}\label{def:wass}
  d_{p}(\mu_1,\mu_2)\hspace{-0.2em}
  = \hspace{-0.5em}
\inf_{\gamma \in \Gamma(\mu_1,\mu_2)}\hspace{-0.2em}
\left(
\mathop{\mathbb{E}}\limits_{(x_1,\, x_2)\sim\gamma}\hspace{-0.4em}
  \left[\, 
    \left\|x_1-x_2\right\|^{p}\,
  \right]
\hspace{-0.1em}\right)^{\frac{1}{p}}
\end{equation}
where $\Gamma(\mu_1,\mu_2)$ is the set of all couplings of $\mu_1$ and $\mu_2$.
For more details, including a proof that this definition indeed satisfies the axioms of a metric, the reader is referred to  Chapter 6 of \citep{villani2008}.

In order to link the distance of the functions
$\nabla f_T$ and
$\nabla f_V$ to the distance between the empirical measures $\mu_T$ and $\mu_V$, the following assumption will be useful:
  \begin{asu}\label{asu:empirical-distance}
    The function $\nabla_x f$ is $G$-Lipschitz continuous as a function of $y$: $ \forall x \in \reals^d, y_1,y_2 \in \reals^q,$
    \[\|\nabla_x f(y_1,x) - \nabla_x f(y_2,x)\| \leq G \|y_1 - y_2\|. \]
  \end{asu}
Assumption \ref{asu:empirical-distance} implies the following bound:   $\forall x\in\reals^d,$
  \begin{equation}\label{grad-dist}
    \|\nabla f_V(x) - \nabla f_{T}(x)\| \leq G d_1(\mu_V,\mu_T).
    \end{equation}
 To see that  \eqref{grad-dist} follows from Assumption \ref{asu:empirical-distance}, let $\gamma$ be any coupling of $\mu_V$ and $\mu_T$. Then
 \begin{align*}
   \left\|\nabla f_{V}(x)\hspace{-0.2em} -\hspace{-0.2em} \nabla f_{T}(x)\hspace{-0.1em}\right\|\hspace{-0.2em}&=\hspace{-0.2em}
   \left\|
   \mathop{\mathbb{E}\hspace{1em}}\limits_{(u,v) \sim \gamma}
          \hspace{-0.45em}[ \nabla_x f(u,x\hspace{-0.1em})\hspace{-0.2em} -\hspace{-0.2em} \nabla_x f(v,x\hspace{-0.1em})]\hspace{-0.0em}\right\|
                  \\&\leq
                  G \mathop{\mathbb{E}}\limits_{(u,v) \sim \gamma}
                  \left[             \left\| y_1-y_2\right\| \right]. 
 \end{align*}
 Taking the infimum over all couplings of $\mu_V$ and $\mu_T$ yields Equation \eqref{grad-dist}.
 For  an example of a function that satisfies
 Assumption \ref{asu:empirical-distance}, consider the following:
  \begin{example}\label{ex:lip}
    Let $g :\reals^q \times \reals^d \to \reals$ be any smooth (that is, infinitely differentiable) function, and let
  $h : \reals^d \to \reals^d$ 
  be the function that applies the hyperbolic tangent function to each of its components: 
  $h(x) = (\tanh(x_1),\hdots, \tanh(x_d))$.
  Define
  $f(y,x) = g( y, h(x))$,
  and further suppose that the training data are bounded: 
  $\|y\| \leq J$ for all $y \in Y_T \cup Y_{V}$.
  To guarantee that
  Assumption \ref{asu:empirical-distance} is satisfied, it is sufficient that the derivative
  $\frac{\partial^{2} f}{\partial x\partial y}(y,x)$
  is  bounded as a bilinear map, uniformly in $y$ and $x$ (Proposition 2.4.8 in \citep{abraham2012manifolds}).
  It can be shown that this is indeed the case, and we may take $G = \sup_{\|y\| \leq J, \|x\| \leq \sqrt{d}}\|\tfrac{\partial ^2 g}{\partial x\partial y}(y,x)\|$. We defer the details to an appendix.
  \end{example}
  In our analyses the notion of success is that an algorithm generates an approximate stationary point:
  \begin{definition}
    A point $x \in \reals^d$ is an \textit{$\epsilon$-approximate stationary point} of 
    $f$ if $\|\nabla f(x)\|^{2} \leq \epsilon$.
  \end{definition}
 We measure the complexity of algorithms according to how many function value and gradient queries they make. 
Formally, an IFO is defined as follows \citep{agarwal2015lower}: 
\begin{definition}
An IFO takes a parameter $x$ and an input $y$ and returns the pair $(f(y,x), \nabla_{x} f(y,x))$.
\end{definition}

In
Appendix \ref{stoch-proc}, we briefly recall the notion of filtration, stopping times, and other concepts from stochastic processes that will be used in this paper. 
\section{\uppercase{Biased SGD}}\label{sect:bias}
In this section we present our analysis of SGD with early stopping. The steps of the procedure are detailed in in
Algorithm \ref{algo1}. Starting from an initial point $x_1$, the parameter is updated  at each iteration with an approximate gradient $h_n$, using a step-size $\eta$. The norm of the gradient of the validation function is evaluated every $m$ iterations, and the algorithm ends when the squared norm of the gradient falls below a threshold $\epsilon$.

We assume that the update direction $h_t$ is a sum of two components, $v_t$ and $\Delta_t$, that represent an unbiased gradient estimate and an error term, respectively:
  \begin{equation}\label{w-form}
    h_{t}
    =
    v_{t}
    +
    \Delta_{t}.
  \end{equation}
  Let $\{\mc{F}_{t}\}_{t\geq 0}$
  be a filtration such that $x_1$ is $\mc{F}_{0}$-measurable,
  and
  for all $t>1$, the variables $(v_t,\Delta_t)$ are $\mc{F}_{t}$-measurable.
  Our assumptions on the $v_t$ are as follows.
\begin{asu}\label{asu:vanilla-grad}
 For any $t \geq 1$, it holds that
  \begin{align} 
    &\mathbb{E}\left[v_{t} - \nabla f_{T}(x_t) \mid \mc{F}_{t-1}\right] = 0,  \label{delta-prop-1} \\
    &\mathbb{E}\left[
    \left\|v_{t} - \nabla f_{T}(x_t)\right\|^{2} \mid \mc{F}_{t-1}
    \right] \leq \sigma_v^2. \label{delta-prop-2}
    \end{align}
\end{asu}
\begin{figure}
\begin{minipage}{0.48\textwidth}\vspace{-1em}
\begin{algorithm}[H]
  \caption{SGD with early stopping\label{algo1}}
\fw{\phantom{ }1:}\,   \textbf{input:} Initial point $x_1 \in \reals^d$ \\
\fw{\phantom{ }2:}\, $t = 1$  \\
\fw{\phantom{ }3:}\, /* check if stopping criteria is satisfied. */ \\
\fw{\phantom{ }4:}\, \textbf{while} $\|\nabla f_{V}(x_t)\|^2 > \epsilon$ \textbf{do} \\
\fw{\phantom{ }5:}\, \quad /* if not, perform an epoch of training. */ \\
\fw{\phantom{ }6:}\, \quad \textbf{for} $n=t$ \textbf{to} $t+m-1$ \textbf{do} \\
\fw{\phantom{ }7:}\, \quad\quad  $x_{n+1} = x_{n}  - \eta h_n$ \label{line:gd-update} \\
\fw{\phantom{ }8:}\,  \quad  \textbf{end} \\
\fw{\phantom{ }9:}\, \quad $t = t + m$ \\
\fw{10:}\,   \textbf{end} \\
\fw{11:}\, /* once criteria is met, return current iterate. */ \\
\fw{12:}\, \textbf{return} $x_{t}$
\end{algorithm}
\end{minipage}
\end{figure}
Assumption \ref{asu:vanilla-grad} states that the update terms $v_t$ are valid approximations to the gradient,
and have bounded variance.
For the bias terms
  we assume the following growth condition:
  \begin{asu}\label{asu:bias-grad}
    There is a sequence of random variables
    $V_1,V_2,\hdots$, and
    $U_1,U_2,\hdots$  such that for all $t\geq 1$ the pair $(V_t,U_t)$ is $\mc{F}_{t}$-measurable,
     $\|\Delta_t\|^{2} \leq V_t,$
  and the
  $V_t$
  satisfy the following geometric drift condition:
  For some pair of constants $\alpha \in [0,1)$ and $\beta \geq 0$,
    \begin{align}
      & V_1 \leq \beta,  \label{v1-eqn}\\
      \forall t \geq 2 ,\quad & V_{t} \leq \alpha V_{t-1} + U_{t-1}, \label{vn-eqn} \\
      \forall t \geq 1 ,\quad &  \mathbb{E}\left[U_{t} \mid \mc{F}_{t-1}\right] \leq \beta \label{bn-eqn}. 
    \end{align}
  \end{asu}
  Assumption \ref{asu:bias-grad} models a scenario where the bias dynamics  are  a combination of contracting and expanding behaviors. Contraction shrinks the error and is  represented by a factor $\alpha$. External noise, represented by the $U_t$ terms, prevents the error from vanishing completely. Note that the assumption would be satisfied in the unbiased case by simply setting $V_t= 0$.  
    
  We can now state our result on the expected number of iterations required by biased SGD with early stopping:
  \begin{prop}\label{prop:abstract-bias}
    Let $\{x_t\}_{t\geq 1}$ be as in Algorithm \ref{algo1}. Let
    Assumptions 
    \ref{asu:lipgrad}, 
    \ref{asu:finite-sum-fns},
    \ref{asu:empirical-distance},
    \ref{asu:vanilla-grad}, and 
    \ref{asu:bias-grad} hold.    
    For
    $\epsilon > 0$, let
    $\tau(\epsilon)$ be the stopping time
    $$
    \tau(\epsilon)
    \hspace{-0.2em}=\hspace{-0.2em}
    \inf\{
    n \hspace{-0.2em}\geq\hspace{-0.2em} 1 \hspace{-0.2em}
    \mid
    \hspace{-0.2em} n\hspace{-0.1em} \equiv\hspace{-0.1em}1\hspace{-0.2em} \Mod m \hspace{-0.1em}\text{ and } \|\nabla f_V(x_n)\|^{2}\hspace{-0.2em} \leq\hspace{-0.1em} \epsilon\hspace{-0.1em} \}.
    $$
    Suppose that $\eta \leq \frac{1}{L}$ and 
    $$
    \epsilon
    -
    4 L m \eta\sigma_v^2
    -
    4 m \beta/(1-\alpha) - 2 G^2 d_1(\mu_V,\mu_T)^2 > 0.
    $$
    Then
       \begin{equation*}\label{claimed-bias}
      \begin{split}
        &\mathbb{E}[\tau(\epsilon)]
        \leq \\&
        \frac{
 G^2 d_1(\mu_V,\hspace{-0.1em}\mu_T)^2 \hspace{-0.1em} +\hspace{-0.1em} 2(f_T(x_1\hspace{-0.1em})\hspace{-0.1em} -\hspace{-0.1em} f^{*})/\eta\hspace{-0.1em} +\hspace{-0.1em} \epsilon\hspace{-0.1em}+\hspace{-0.1em}2 \beta /(1\hspace{-0.1em}-\hspace{-0.1em}\alpha)}
           {\epsilon/(2 m) - 2 L\eta\sigma_v^2 -2 \beta /(1-\alpha) - G^2 d_1(\mu_V,\mu_T)^2/m}.
               \end{split}
    \end{equation*}
    Furthermore, the gradient of $f_T$ at $x_{\tau(\epsilon)}$ satisfies
    \begin{equation}\label{holder-nabla}
      \|\nabla f_{T}(x_{\tau(\epsilon)})\|^2 \leq
      \left(\sqrt{\epsilon} + Gd_1(\mu_V,\mu_T)\right)^2.
      \end{equation}
  \end{prop}
We present a sketch of the proof below, saving the full proof for an appendix.
\begin{proof}[Proof sketch]
  To emphasize the main ideas, we make the simplifying assumptions that there are no error terms ($\Delta_t=0$), the Lipschitz constant for the gradient is $L=2$, and the training  and validation sets are equal ($Y_T=Y_V$).
   To establish a bound on  $\mathbb{E}[\tau(\epsilon)]$, we first consider the truncated stopping time $\tau(\epsilon) \wedge n$, defined as the minimum of $\tau(\epsilon)$ and an arbitrary iteration number $n$. We  find an upper bound on $\mathbb{E}[\tau(\epsilon) \wedge n]$ that is independent of $n$, and appeal to the monotone convergence theorem to conclude that this same bound must hold for $\mathbb{E}[\tau(\epsilon)]$.
   
   Using a quadratic growth bound  that follows from the Lipschitz property of the gradient (Equation \eqref{eqn:taylor} in the appendix), for any $n$ it holds that
   \begin{equation*}
     \begin{split}
       &f(x_{\tau(\epsilon) \wedge n+1}) \leq 
       f(x_1) - \eta(1-\eta)\sum\limits_{t=1}^{\tau(\epsilon) \wedge n}\|\nabla f(x_{t})\|^2  \\
       &\quad-\eta(1- 2 \eta) \sum\limits_{t=1}^{\tau(\epsilon)\wedge n}\nabla f(x_{t})^T\delta_t +
   \eta^2 \sum\limits_{t=1}^{\tau(\epsilon)\wedge n}
   \|\delta_t\|^2.
     \end{split}
     \end{equation*}
   Taking expectations and applying Proposition \ref{mart-like-thm}, we obtain 
   \begin{equation*}
     \begin{split}
       \mathbb{E}\hspace{-0.1em}\left[f(x_{\tau(\epsilon) \wedge n + 1})\right]
       \hspace{-0.1em}\leq &f(x_1)\hspace{-0.1em}-\hspace{-0.1em}\eta(1\hspace{-0.2em}-\hspace{-0.2em}\eta) \mathbb{E}\hspace{-0.3em}\left[\hspace{-0.1em}\sum\limits_{t=1}^{\tau(\epsilon) \wedge n}\hspace{-0.3em}\|\nabla f(x_{t})\|^2\hspace{-0.1em}\right] \\&+
   \eta^2\sigma_v^2\,\mathbb{E}[\tau(\epsilon)\wedge n].
   \end{split}
   \end{equation*}
   Rearranging terms and noting that $f(x_{\tau(\epsilon) \wedge n + 1}) \geq f^{*}$,
   \begin{equation*}
     \begin{split}
&\eta(1-\eta)\mathbb{E} \left[ \sum\limits_{t=1}^{\tau(\epsilon)\wedge n}\|\nabla f(x_t)\|^{2}\right] \leq\\ &\quad\quad f(x_1) - f^*  + \eta^2 \sigma_v^2\,\mathbb{E}[\tau(\epsilon) \wedge n].
     \end{split}
     \end{equation*}
 Next, using the definition of $\tau(\epsilon)$,  we have
 \begin{equation*}
   \begin{split}
     \frac{\epsilon\left(\mathbb{E}\hspace{-0.1em}\left[\hspace{-0.1em}
       \tau(\epsilon)\wedge n\right]\hspace{-0.1em}-\hspace{-0.2em}1\right)}{m}\hspace{-0.2em}&\leq\hspace{-0.1em}\mathbb{E}\hspace{-0.2em}\left[\hspace{-0.1em}\sum\limits_{t=1}^{\tau(\epsilon)\wedge n}\hspace{-0.3em}1_{t\equiv 1 \Mod m}\|\nabla f(x_{t})\|^2\hspace{-0.1em}\right] \\
     &\leq \mathbb{E}\hspace{-0.2em}\left[\sum\limits_{t=1}^{\tau(\epsilon)\wedge n}\|\nabla f(x_{t})\|^2\right].
   \end{split}
   \end{equation*}
 Combining the previous two equations, upon rearranging terms we obtain
 $$
 \eta\hspace{-0.1em}\left(\hspace{-0.1em} (1\hspace{-0.1em}-\hspace{-0.1em}\eta)\hspace{-0.1em}\frac{\epsilon}{m }\hspace{-0.1em}-\hspace{-0.1em}\eta  \sigma_v^2\hspace{-0.1em}\right)\hspace{-0.2em}\mathbb{E}[\tau(\epsilon) \wedge n]\hspace{-0.2em} \leq\hspace{-0.2em}f(x_1) - f^* +  \frac{\eta(1-\eta)\epsilon}{m}
 $$
 The coefficient on the left hand side of this equation is positive provided that $$\eta < \frac{\epsilon}{m\sigma^2 + \epsilon}$$
 Choose a $c\in(0,1)$ and let $\eta = c\cdot \frac{\epsilon}{m\sigma^2 + \epsilon}$. Rearranging terms, and letting $n\rightarrow \infty$, we obtain
 $$\mathbb{E}[\tau(\epsilon)] \leq  \frac{(f(x_1)- f^*)m^2\sigma^2}{\epsilon^2 c (1-c)} + \mathcal{O}\left(\frac{m}{\epsilon}\right).$$
We refer the reader to the appendix for a complete proof.
 \end{proof}
Note that the condition on $\eta$ in the proposition requires that it scales inversely with the epoch length $m$. Whether this argument can be refined to yield conditions on $\epsilon$ that are independent of $m$, we leave as an open question.
   Let us note that the situation is somewhat more favorable in the case of SVRG. In our analysis of SVRG below, the introduction of early stopping does not produce any new constraints on the step-size.

   Proposition \ref{prop:abstract-bias} implies that SGD can  find $\epsilon$-approximate stationary points, for any $\epsilon > 4 m \beta/(1-\alpha) + 2 G^{2}d_1(\mu_V,\mu_T)^2$. We can relax this condition, allowing for smaller values of $\epsilon$, by assuming a coupling between
the step-size and the expansion coefficient,  as demonstrated in the next corollary.
  \begin{cor}\label{cor:bias-cvg}
      Let Assumptions \ref{asu:lipgrad},   \ref{asu:finite-sum-fns}, \ref{asu:empirical-distance}, \ref{asu:vanilla-grad}, and \ref{asu:bias-grad} hold.    
    In the context of Proposition \ref{prop:abstract-bias}, let the constant $\beta$ be of the form
    $\beta = \eta R$
    for some $R \geq 0$, and suppose that $\epsilon > 2 G^2 d_1(\mu_V,\mu_T)^2$.
    Let
    $c \in (0,1)$ and
    let the step-size be
    \begin{equation}\label{step-form}
    \eta
    =
    c
    \cdot
    \min
    \left\{
    \frac{1}{L},
    \frac{  \epsilon/2 - G^2 d_1(\mu_V,\mu_T)^2}{m ( 2 L \sigma_v^2 + 2 R / ( 1 - \alpha )) }
    \right\}.
    \end{equation}
    Then 
    \begin{equation*}
      \begin{split}
          \mathbb{E}[\tau(\epsilon)] = \mathcal{O}\left(
              \frac{m^2 \left( 1 +  R /(1-\alpha)\right)}{(1-c)\,c\,(\epsilon - 2 G^2 d_1(\mu_V,\mu_T)^2)^2}\right).
           \end{split}
      \end{equation*}
  \end{cor}  
The reader may refer to the full proof contained in an appendix for the complete formula, including lower order terms. This result will be used below, in our analysis of DSGD.

   We now specialize the results in the case of using SGD to minimize a finite sum using unbiased gradient estimates.

 \begin{cor}\label{sgd-ifo}
   Let Assumptions \ref{asu:lipgrad},
   \ref{asu:finite-sum-fns}, \ref{asu:ft-var-bias}, \ref{asu:empirical-distance} hold.
    Suppose each gradient estimate is obtained by selecting an element
    $y_{t} \in Y_{T}$ 
    uniformly at random and setting 
    $v_t = \nabla_x f(y_t,x_t)$.
    Let $\epsilon > 2 G^{2} d_1(\mu_V,\mu_T)^2$ and consider running SGD with
     epoch length $m\geq 1$, and step-size $\eta$ as defined in \eqref{step-form} with $c=1/2$.
    Then the expected number of IFO calls used by SGD with early stopping is
    \begin{equation*}
      \begin{split}       
    &\mathbb{E}\left[\mathrm{IFO}\left(\epsilon\right)\right]
        = 
   \mc{O}\left(
   \frac{
    m n_V + m^2
   }
   {(\epsilon -  2 G^2 d_1(\mu_V,\mu_T)^2)^2}
   + n_V\right) \\
      \end{split}
      \end{equation*}
    \end{cor}
Note that when $d_1(\mu_V,\mu_T)=0$, this result  states that the expected IFO complexity is $\mc{O}(1/(\epsilon^2))$. This can be compared with the RSG algorithm, where $\mc{O}(1/(\epsilon^2))$ iterations are sufficient for the expected squared norm of the gradient at a random iterate to be at most $\epsilon$ (Corollary  2.2 in \citep{ghadimi-lan}).

\section{\uppercase{Decentralized SGD}}\label{sect:dsgd}
In this section we analyze the expected running time of decentralized SGD (DSGD), a variant of SGD designed for distributed optimization across a network of compute nodes. Recently, a randomization-based analysis of DSGD was presented in \citep{decentralized}. We complement that analysis by studying the expected running time of a variant of DSGD  with early stopping. The main idea is to model the algorithm as a biased form of SGD that satisfies the geometric drift condition described in Assumption \ref{asu:bias-grad}.

\begin{figure}
\begin{minipage}{0.48\textwidth}\vspace{-1em}
\begin{algorithm}[H]
  \caption{DSGD with early stopping\label{algodgd}}
\fw{\phantom{ }1:}\,   \textbf{input:} Node id $i$, initial parameters $x_{1}^{i}$.  \\
\fw{\phantom{ }2:}\, $t = 1$  \\
\fw{\phantom{ }3:}\, /* check if stopping criteria is satisfied. */ \\
\fw{\phantom{ }4:}\, \textbf{while} $\|\nabla f_{V}(\overline{x}_t)\|^2 > \epsilon$ \textbf{do}\\
\fw{\phantom{ }5:}\, \quad /* if not, perform an epoch of training. */ \\
\fw{\phantom{ }6:}\, \quad \textbf{for} $n=t$ \textbf{to} $t+m-1$ \textbf{do} \\
\fw{\phantom{ }7:}\, \quad \quad /* perform local average and descent steps. */\\
\fw{\phantom{ }8:}\, \quad \quad $x_{n+1}^{i} = \sum\limits_{j=1}^{M}a_{i,j} x_{n}^{j} - \eta v_{n}^{i}$ \\
\fw{\phantom{ }9:}\, \quad  \textbf{end} \\
\fw{10:}\, \quad $t = t + m$ \\
\fw{11:}\,   \textbf{end} \\
\fw{12:}\, /* once criteria is met, return current iterate. */ \\
\fw{13:}\, \textbf{return} $\overline{x}_t$
\end{algorithm}
\end{minipage}
\end{figure}

The steps of DSGD are shown in Algorithm \ref{algodgd}. The procedure involves $M > 0$ worker nodes that participate in the optimization, and  an $M\times M$ communication matrix $a$ describing the connectivity among the workers; $a_{i,j} > 0$ means that workers $i$ and $j$ will communicate at each iteration, while $a_{i,j} = 0$ means there is no direct communication between those workers. At each step of optimization, every node computes a weighted average of the parameters in its local neighborhood, as determined by the connectivity matrix. This is combined with a local gradient approximation to obtain the new parameter at the worker. Every $m$ epochs, the norm of gradient of the validation function is evaluated at the average parameter across the system, denoted $\overline{x}_t$:
\begin{equation}\label{sys-avg}
  \overline{x}_t = \frac{1}{M}\sum\limits_{i=1}^{M}x_{i}
\end{equation}
When this norm falls below a threshold,  the algorithm terminates, returning the final value of $\overline{x}_t$.

The intuitive justification for DSGD is that it may be more efficient compared to naive approaches to parallelizing SGD, since two nodes $i$ and $j$ need not communicate when $a_{i,j} = 0$. In  \citep{decentralized} those authors offer theoretical support for the superiority of DSGD. In the present work, our goal is to analyze the expected running time of DSGD as an example of how the theory developed above may be applied in practice.

For the analysis, define the filtration $\{ \mc{F}_{t} \}_{t\geq 0}$ as follows:
\begin{align*} 
  \mc{F}_{0} &=
  \sigma\big(
  \left\{ x_{1}^{i} \,\middle|\, 1 \leq i \leq M \right\}\big), \\
  \forall t\geq 1, \quad
  \mc{F}_{t}
  &=
  \sigma
  \big(
  \left\{ x_{1}^{i}, v_{n}^{i} \,\middle|\,  1 \leq n \leq t,\, 1\leq i \leq M \right\}\big).
\end{align*}
We assume that the gradient estimates used at each worker are unbiased and have bounded variance.
\begin{asu}\label{asu:decentralized}
  For any $t\geq 1$ and $1\leq i \leq M,$
  \begin{align}
   &\mathbb{E}\left[v_{t}^{i} - \nabla f_T(x_{t}^{i}) \mid \mc{F}_{t-1} \right] = 0, \\
    &\mathbb{E}\left[
      \left\| v_{t}^{i} - \nabla f_{T}(x_{t}^{i}) \right\|^{2} \mid \mc{F}_{t-1}
      \right]
    \leq
    \sigma_v^2.\label{eqn:decentra-var-bd}
  \end{align}
\end{asu}

The connectivity matrix $a$ is subject to the same conditions as in \citep{decentralized}, stated below as Assumption \ref{asu:connectivity}. In this Assumption, $\lambda_i(a)$ refers to the eigenvalues of the matrix $a$ in nonincreasing order:
$ \lambda_i(a) \geq \lambda_{i+1}(a)$  for $1\leq i < M$.
\begin{asu}\label{asu:connectivity}
  The  $M\times M$ connectivity matrix, denoted $a$, is symmetric and stochastic.
  The diffusion coefficient, denoted by $\rho$ and  defined as $\rho =  \max_{2\leq i\leq M} |\lambda_i(a)|^{2}$, satisfies $\rho <1$.
\end{asu}

We will show that the sequence of averages $\overline{x}_{t}$ for $t=1,2,\hdots$ can be modeled in terms of biased SGD, using the tools from Section \ref{sect:bias}. This involves showing that the distance between local parameter values  and the system average obey a geometric drift condition, and furthermore, this distance can be controlled through the step-size.
\begin{prop}\label{decentralized-dispersion}
  Let
  Assumptions
  \ref{asu:lipgrad},
  \ref{asu:finite-sum-fns},
  \ref{asu:decentralized}, and
  \ref{asu:connectivity} hold, and
  let the step-size satisfy
  \begin{equation}\label{decentra-eta-cond}
  \eta
  \leq
  \frac{1-\sqrt{\rho}}{4 L\sqrt{2}}.
  \end{equation}
  Define the variables $V_1,U_1, V_2,U_2, \hdots$ and the constants $\alpha,\beta$  as follows:
  \begin{subequations}
    \begin{align}
      V_{t} &=
      \frac{L^2}{M}\sum\limits_{i=1}^{M}\|x_{t}^{i} - \overline{x}_t\|^2, \label{decentra-v} \\
      U_{t} &=
       \frac{32 \, \eta^2 \,L^2}{M(1-\sqrt{\rho})}\sum\limits_{i=1}^{M}
      \|v_t^i - \nabla f(x_t^i)\|^2,\label{decentra-u}\\
      \alpha &= \frac{\left(3 +\sqrt{\rho}\right)^2}{16}, \label{decentra-alpha} \\
      \beta &= \eta \frac{8 L }{1-\sqrt{\rho}}\sigma^2_v. \label{decentra-beta}
    \end{align}
  \end{subequations}
  Then for all $t\geq 1$ it holds that
  $V_{t+1} \leq \alpha V_t + U_t$ and
  $\mathbb{E}[U_t \mid \mc{F}_{t-1}] \leq \beta$.
  \end{prop}
We can now move to the main result on decentralized SGD. The result gives conditions that guarantee the expected time
$\mathbb{E}[\tau(\epsilon)]$
is finite, and also bounds this time in terms of the problem data, including  the epoch length, variance, and the mixing rate of the connectivity matrix.
\begin{prop}\label{decentralized-main}
  Let
  Assumptions
  \ref{asu:lipgrad},
  \ref{asu:finite-sum-fns},
  \ref{asu:empirical-distance},
    \ref{asu:connectivity},
and
  \ref{asu:decentralized} hold,
  and assume that the initial parameters at every node are equal:
  $x_{1}^{i} = x_{1}^{j}$
  for all  $ 1 \leq i,j \leq M$.
  Suppose that $\epsilon > 2 G^2 d_1(\mu_V,\mu_T)^2$. Let
    $c \leq (1-\sqrt{\rho})/(4\sqrt{2})$,
    and let the step-size be
    \begin{equation*}
    \eta
    \hspace{-0.1em}=\hspace{-0.2em}
    \frac{c}{L}\hspace{-0.1em}
    \min\hspace{-0.1em}
    \left\{\hspace{-0.2em}
    1,\hspace{-0.2em}
    \frac{  \epsilon/2 - G^2 d_1(\mu_V,\mu_T)^2 }
         {
           2 m \sigma_v^2
           (1\hspace{-0.2em}+\hspace{-0.2em}
           128/(7\hspace{-0.1em}+\hspace{-0.1em}5\rho\hspace{-0.1em}+\hspace{-0.1em}\rho^{3/2}\hspace{-0.2em}-\hspace{-0.2em}13\sqrt{\rho})) }\hspace{-0.2em}
    \right\}\hspace{-0.2em}.
    \end{equation*}
    If $\tau(\epsilon)$ is the stopping time
    $$
    \tau(\epsilon)
    \hspace{-0.2em}=\hspace{-0.2em}
    \inf\{
    n \hspace{-0.2em}\geq\hspace{-0.2em} 1 \hspace{-0.2em}
    \mid
    \hspace{-0.2em} n\hspace{-0.1em} \equiv\hspace{-0.1em}1\hspace{-0.2em} \Mod m \hspace{-0.1em}\text{ and } \|\nabla f_V(\overline{x}_n)\|^{2}\hspace{-0.2em} \leq\hspace{-0.1em} \epsilon\hspace{-0.1em} \}.
    $$
then
    \begin{equation*}
      \begin{split}
        &\mathbb{E}[\tau(\epsilon)] =\\&\quad \mathcal{O}\left(
          \frac{m^2}
               {
                 (1\hspace{-0.1em}-\hspace{-0.1em}c)\,c\,
                 (\epsilon\hspace{-0.1em} -\hspace{-0.1em} 2 G^2 d_1(\mu_V,\mu_T)^2)^2
                 (1\hspace{-0.1em}-\hspace{-0.1em} \sqrt{\rho})^2
               }\right).
           \end{split}
      \end{equation*}
\end{prop}
Note that in the above result, the order of the convergence is the same as for regular SGD.
  
 Using these tools allows us to bound the expected number of IFO calls needed by DSGD to find approximate stationary points.
 \begin{cor}\label{cor:decentralized-ifo}
    Let Assumptions \ref{asu:lipgrad}, \ref{asu:finite-sum-fns}, \ref{asu:ft-var-bias} and \ref{asu:empirical-distance} hold. 
    Suppose each gradient estimate is obtained by selecting an element
    $y^j_t \in Y_T$
    uniformly at random and setting
    $v^j_t = \nabla_x f(y^j_t, x_t^j)$.
    Let
    $\epsilon > 2 G^2 d_1(\mu_V,\mu_T)^2$ and
    consider running DSGD with
    epoch-length
    $m \geq 1$, and step-size
    $\eta$ as defined in Proposition \ref{decentralized-main} with
    $c = (1-\sqrt{\rho})/(4\sqrt{2})$.
    Then  the expected number of IFO calls used by DSGD with early stopping is 
    \begin{equation*}
      \begin{split}
    &\mathbb{E}\left[\mathrm{IFO}(\epsilon)\right] 
    = \\&\quad
    \mc{O}\left(
    \frac{m (n_V + m M)
    }{(1-\sqrt{\rho})^3\sqrt{\rho}
      \left(\epsilon - 2 G^2 d_1(\mu_V,\mu_T)^2\right)^2}+ n_V
    \right).
      \end{split}
      \end{equation*}
  \end{cor}
Note the factor of $M$ that appears in the numerator. This is due to the fact that $M$ gradients are evaluated at each iteration of the algorithm, one at each node. 
\section{\uppercase{SVRG}}
  \begin{figure}\vspace{-1em}
  \noindent\begin{minipage}{0.48\textwidth}
      \begin{algorithm}[H]
  \caption{SVRG with early stopping\label{algo:svrg}}
\fw{\phantom{ }1:}\, \textbf{input:} Initial point $x_m^1 \in \reals^d$   \vspace{0.2em}  \\ 
\fw{\phantom{ }2:}\, \textbf{for} $s = 1,2,\hdots$ \textbf{do}    \vspace{0.2em}     \\
\fw{\phantom{ }3:}\, \quad   $x_0^{s+1} = x_{m}^{s}$  \label{line:svrg2}   \vspace{0.2em}  \\
\fw{\phantom{ }4:}\, \quad  
    \(  g^{s+1} =
     \tfrac{1}{n_T}
     \sum_{y \in Y_T}
     \nabla_x f(y,x_0^{s+1})
     \) \label{line:svrg3} \vspace{0.2em}\\
\fw{\phantom{ }5:}\, \quad\textbf{if} $\|g^{s+1}\|^2 \leq \epsilon$ \textbf{then} \textbf{return} $x_{0}^{s+1}$  
\vspace{0.2em}
     \\
\fw{\phantom{ }6:}\,  \quad\textbf{for} $t=0$ \textbf{to} $m-1$ \textbf{do}
 \vspace{0.2em}
 \\
\fw{\phantom{ }7:}    \quad\quad Sample $y_{t}^{s}$ uniformly at random from $Y_T$ \label{line:svrg4}    
 \vspace{0.2em}
 \\
\fw{\phantom{ }8:}   \quad\quad  
   $v_{t}^{s}\hspace{-0.1em}=\hspace{-0.2em}
    \nabla_x f(y_{t}^{s},x_{t}^{s+1})
    \hspace{-0.1em}-\hspace{-0.1em}
    \nabla_x f(y_{t}^{s},x_{0}^{s+1})
    +
    g^{s+1}$ \label{line:svrg5}
    \vspace{0.2em}
    \\
\fw{\phantom{ }9:}   \quad\quad 
      $x_{t+1}^{s+1} = x_{t}^{s+1}  - \eta v_{t}^{s}$ \label{line:svrg6}
           \vspace{0.2em}
           \\
\fw{10:}\, \quad \textbf{end}  
      \vspace{0.2em}
      \\
\fw{11:}\,   \textbf{end} \vspace{0.2em}
\end{algorithm}
\end{minipage}%
  \end{figure}

In this section we analyze a variant of SVRG \citep{svrg} with early stopping. The steps of the procedure are shown in Algorithm \ref{algo:svrg}.
  Each epoch begins with a full gradient computation (Line 4). Next, the norm of the  gradient is computed, and if it falls below the threshold $\epsilon$, the algorithm terminates, returning the current iterate. Otherwise, an inner loop runs for $m$ steps.
  The first step of the inner loop is to choose a random data point (Line 7). Then, the update direction is computed (Line 8) and used to obtain the next parameter (Line 9). 
  
The technical tools we use to analyze SVRG with early stopping include existing bounds for SVRG \citep{reddi2016stochastic} along with the optional stopping theorem.  Together, they yield the following bound on the expected number of epochs until SVRG with early stopping terminates.
  \begin{prop}\label{prop:svrg-conv}
     Let Assumptions \ref{asu:lipgrad} and \ref{asu:finite-sum-fns} hold
      and consider the variables
      $x_{t}^{s+1}$
      defined by Algorithm \ref{algo:svrg}.
      Suppose that the step-size is set to 
      $\eta = 1/( 4 L n_T^{2/3})$
      and the epoch length is 
      $m = \left\lfloor 4 n_T/ 3 \right\rfloor$.
      For 
      $\epsilon >0$, 
      define 
      $\tau(\epsilon)$ to be the stopping time
      $
        \tau(\epsilon)
      =\inf\left\{
      s \geq 1 \,\middle|\, \|\nabla f_T(x_{0}^{s+1})\|^{2} \leq \epsilon 
      \right\}.
      $
      Then
      $$\mathbb{E}[ \tau(\epsilon)] \leq 1 + \frac{40 L n_{T}^{2/3} (f_T(x_{m}^1) - f^{*}) }{\epsilon}.$$
  \end{prop}
 
  Note that Proposition \ref{prop:svrg-conv} counts the number of epochs until an approximate stationary point is generated. A bound on the number of IFO calls can be obtained by multiplying $\tau$  by the number of IFO calls per epoch, which is $n_{T}+2 m$. This immediately leads to the following result:
  
  \begin{cor}\label{svrg-ifo}
    Let Assumptions \ref{asu:lipgrad} and \ref{asu:finite-sum-fns} hold and suppose the step-size $\eta$ and epoch length $m$ are defined as in Proposition \ref{prop:svrg-conv}.
    Then, the expected number of IFO calls until SVRG returns an approximate stationary point is
    $\mathbb{E}
    \left[
      \mathrm{IFO}\left(\epsilon\right)
      \right]
    = \mc{O}((n_T^{5/3}/\epsilon) + n_T )
    $.
  \end{cor}
  This result may be compared with Corollary  4 of \citep{reddi2016stochastic}, which concerns an upper bound on the IFO calls needed for the expected (squared) norm of the gradient at a randomly selected iterate to be less than $\epsilon$. Our result concerns the expected number of IFO calls before the algorithm terminates with an iterate that is guaranteed to be an approximate stationary point.   Note that introducing  early stopping does not add any complexity, compared to SGD. This is because the full gradient is already calculated at each iteration, and the only additional step in the algorithm is computation of the norm.

\section{\uppercase{Generalization Properties}}\label{sect:genprop}
Typically, the training and validation sets are made of independent samples from a test distribution $\mu$, and it is of interest to estimate the model performance relative to this test distribution. Formally, define the generalization error $f_{G}$ as
$f_{G}:\reals^d \to \reals$ as $f_G(x) = \mathbb{E}_{y\sim\mu}[f(y,x)]$. In this section, we consider upper bounds on the quantity
\begin{equation}\label{expect-grad}
  \mathbb{E}
  \left[
    \left\|\nabla f_{G}(x_{\tau(\epsilon)})\right\|^2\right],
\end{equation}
where $x_{\tau(\epsilon)}$ is the iterate returned by an optimization algorithm with early stopping.
Note that this expectation is  over both the variates generated by optimization and the random choice of the datasets $Y_V$ and $Y_T$.  In this section we show how Wasserstein concentration results can be used to bound \eqref{expect-grad}, in terms of both the norm of the gradient of the training function, and the Wasserstein distance between $\mu$ and its empirical version used for optimization.

To begin, note that under Assumption \ref{asu:empirical-distance}, the gradient of the generalization error can be related to the gradient of the training error by
$$
\mathbb{E}[\|\nabla f_{G}(x_{\tau(\epsilon)})\|]
\leq
\mathbb{E}[\|\nabla f_{T}(x_{\tau(\epsilon)})\|]
+
G \mathbb{E}[d_1(\mu_T,\mu)]$$
The second term on the right is the expected distance between the empirical measure $\mu_T$ and the data distribution $\mu$. Intuitively, for large values of $n_T$ the empirical distribution should be a good approximation to the true distribution, and the distance should be small. Investigations into the convergence rate of $d_{p}(\mu,\mu_T)$ as a function of $n_T$ has received significant attention, beginning with \citep{dudley1969}. For more background we refer the reader to \citep{dereich2013},\citep{weedbach2017} and references therein. For our purposes, the basic idea can be illustrated with the following result.
\begin{thm}[\citep{dereich2013}, Theorem 1]\label{quantization} 
  For $d\geq 3$, let $\mu$ be a measure on $\reals^d$, such that 
  $J = \mathbb{E}_{y\sim\mu}\left[\left\|y\right\|^{3}\right]^{1/3} < \infty,$   and let $\mu_{N}$ be an empirical version of $\mu$  constructed from $N$ samples.
Then there is a constant $\kappa_d$  such that
\[
\mathbb{E}\left[d_2\left(\mu,\mu_N\right)^2\right] 
\leq 
\kappa_d
J
N^{-3/d}.
\]
\end{thm}
The constant $\kappa_d$ is explicitly given in (\citep{dereich2013}, Theorem 3). Note the dependence on the dimension $d$ on the right hand side of this bound, which implies a very slow convergence of the empirical distance in high dimensions. Despite this, the bound is asymptotically tight, for large values of $N$.  An example of a distribution that displays convergence of order $N^{-1/d}$ is the uniform distribution on $[0,1)^d$ (for a proof see Theorem 2 in \citep{dereich2013}). In a machine learning context,  this would correspond to a regression problem where there is no relation between the input and output. We note however, that stronger rates of convergence can be obtained for restricted classes of measures, and that for smaller values of $N$ the convergence rate can be more favorable. This is explored in \citep{weedbach2017} where they improve the bounds for a number of classes of distributions. For instance, when $\mu$ is a discrete distribution, the following holds:
\begin{thm}[\citep{weedbach2017}, Proposition 13]\label{discrete-quantization}
  Let $\mu$ be a measure that is supported on at most $m$ points within the unit sphere in $\reals^d$, and let $\mu_N$ be an empirical version of $\mu$ constructed from $N$ samples. Then
  \[
  \mathbb{E}\left[d_{2}\left(\mu,\mu_N\right)^2\right]
  \leq
  84\sqrt{\frac{m}{N}}.
  \]
  \end{thm}

Depending on the properties of the testing distribution, either one of Theorems \ref{quantization} or  \ref{discrete-quantization} can be used to investigate the dependence of the generalization error on the data set size $n_{T}$. This would involve having some prior knowledge about the nature of the testing set.

In the remainder of this section, we consider combining the concentration bounds with the optimization bounds proved for SVRG. Note that the basic ideas can be applied just as well to SGD or DSGD.

For SVRG, it is natural to express the bound in terms of the number of training examples, and we obtain the following
\begin{cor}\label{cor:wass-svr}
    Let Assumption \ref{asu:empirical-distance} and the conditions of Proposition \ref{prop:svrg-conv} hold. Further assume  $J = \mathbb{E}_{y\sim\mu}\left[\|y\|^{3}\right]^{1/3} < \infty$ and the training set $Y_T$  is an empirical version of $\mu$.
    If $x_{\tau}(\epsilon)$ is the output of Algorithm \ref{algo:svrg}, then
    $$
    \mathbb{E}[\|\nabla f_{G}(x_{\tau(\epsilon)})\|^2] \leq 2\epsilon + 2 G^2\kappa_d J n_T^{-3/d}.
    $$
        Alternatively, if $\mu$ is a supported on at most $m$ points, then
    $$
    \mathbb{E}[\|\nabla f_{G}(x_{\tau(\epsilon)})\|^2] \leq 2\epsilon + 168 G^2 \sqrt{\frac{m}{n_T}}.
    $$
\end{cor}
    Together with bounds on the expected running time,  this  result could potentially let one balance between the resources needed to minimize the training function, and the resources needed to gather training data. In order to minimize the right hand side, one can either choose a smaller $\epsilon$, leading to longer running times, or choose a large $n_{T}$, leading to more sampling.
    
 Note that Corollary \ref{cor:wass-svr} is accounts for data distribution properties (via the $3^{rd}$ moment $J$, or via the number of points in the discrete case) and does not depend on the number of iterations used in SGD. This result could be compared with \citep{tfgb}, where the authors proved a bound on the generalization gap for function values in terms of the number of iterations $T$ and the number samples in the training set. There, the bound is increasing with $T$. An interesting avenue for future work would be to investigate the combination of the two approaches.

\section{\uppercase{Discussion}}
This work presented an analysis of several stochastic gradient-based optimization algorithms that use early stopping. Our focus was on procedures that return the first point satisfying a stopping criterion, and we analyzed the expected running time and number of gradient evaluations needed to meet this criterion.

For SGD, we analyzed the use of early stopping with a validation function, and obtained a bound on the expected number of gradient evaluations needed to find approximate stationary points. The analysis allows for biases in the update direction, subject to a geometric drift condition on the error terms. We specialized this analysis to bound the expected running time of decentralized SGD, a distributed variant of SGD. We modeled DSGD as a biased form of SGD, with a bias term that is controlled in part by the mixing coefficient of the communication graph. Next, we turned to a variant of nonconvex SVRG that employs early stopping, obtaining a bound on the expected number of IFO calls and gradient evaluations used by the algorithm. Lastly, we considered how Wasserstein concentration bounds can be leveraged to bound the generalization performance of the iterate returned by the algorithms, expressed in terms of the number of samples used to define the input datasets, and properties of the data distribution.

We would like to highlight two avenues for future work. Our analysis of SGD has a condition on the step-size that depends on the epoch length $m$ (Corollary \ref{cor:bias-cvg}). It is an interesting question whether this requirement can be removed. Secondly, in our analysis of SVRG, introducing early stopping let to a convergence bound that is essentially the same as the rate obtained using randomization. For SGD, the expected number of IFO calls increases quadratically with the epoch length, and we leave it as an open question whether this is feature can also be relaxed.

\section*{ACKNOWLEDGEMENTS}
This material is based upon work supported by the U.S. Department of Energy, Office of Science, under contract number \textsc{DE\text{-}0012704}.
\section*{REFERENCES}
\renewcommand{\bibsection}{}
\bibliography{super}
\bibliographystyle{icml2019}
\clearpage
\onecolumn
\section*{Appendix: Bounding the expected run-time of nonconvex optimization with early stopping}

\setcounter{section}{0}
\renewcommand{\thesection}{\Alph{section}}
\section{Preliminaries}
 Our analyses make use of a quadratic bound for the training function which follows from Assumption \ref{asu:lipgrad}:
  \begin{equation}\label{eqn:taylor}
    \forall x,v \in \reals^n, \quad f_T(x+v) \leq f_T(x) + \nabla f_T(x)^{T}v + \frac{L}{2}\|v\|^{2}.
  \end{equation}
  For a derivation of Equation \eqref{eqn:taylor}, see for instance Lemma 1.2.3 in \cite{nesterov2013introductory}.
  \section*{Stochastic processes \label{stoch-proc}}
  The formal setting of a stochastic optimization algorithm involves a probability space
  $(\Omega,\mc{F},\mathbb{P})$, consisting of a sample space
  $\Omega$, a
 $\sigma$-algebra $\mc{F}$ of subsets of $\Omega$ and a probability measure $\mathbb{P}$ on the subsets of $\Omega$ that are in $\mc{F}$.  The algorithm takes an initial point $x_1$ and defines a sequence of random variables $\{x_t(\omega)\}_{t>1}$. Intuitively $\Omega$ represents the random data used by the algorithm, such as  indices used to define  mini-batches. For ease of notation we will omit the dependence of random variates in the algorithms on 
 $\omega \in \Omega$.  A filtration $\{\mc{F}_{t}\}_{t=0,1,\hdots}$ is an increasing sequence of $\sigma$-algebras, with the interpretation that $\mc{F}_{t}$ represents the information available to an algorithm up to and including time $t$.  
A random variable $x: \Omega \to \reals^d$ is said to be $\mc{F}_{t}$ measurable if it can be expressed in terms of the state of the algorithm up and including time $t$. A rule for stopping an algorithm is represented as a stopping time, which is a random variable $
\tau :\Omega \to \{0,1,\hdots, \infty\}$ 
with the property that the decision of whether to stop or continue at time $n$ is only made based on the information available up to and including time $n$.

The following proposition will be used through out our analysis of the different algorithms.
\begin{prop}\label{mart-like-thm}
Let
$\tau$
be a stopping time with respect to a filtration
$\{\mc{F}_t \}_{t=0,1,\hdots}$.
Suppose there is a number
$c <\infty$
such that
$\tau \leq c$
with probability one.
Let
$x_1,x_2,\hdots$
be any sequence of random variables  such that each $x_t$ is $\mc{F}_{t}$-measurable and
$\mathbb{E}[\|x_t\|] < \infty$.
Then
\begin{equation}\label{stopping}
\mathbb{E}\left[\sum\limits_{t=1}^{\tau}x_t\right]
=
\mathbb{E}\left[\sum\limits_{t=1}^{\tau}\mathbb{E}\left[x_t \mid \mc{F}_{t-1} \right]\right].
\end{equation}
\end{prop}
\begin{proof}
We argue that \eqref{stopping} is a consequence of the optional stopping theorem (Theorem 10.10 in \citep{williams1991probability}).
Define $S_0= 0$ and for $t\geq 1$, let
$S_t =\sum\limits_{i=1}^{t}\left(x_i - \mathbb{E}[x_i \mid \mc{F}_{i-1}]\right)$.
Then
$S_0,S_1,\hdots$
is a martingale with respect to the filtration
$\{\mc{F}_{t}\}_{t=0,1,\hdots}$,
and the optional stopping theorem implies
$\mathbb{E}[S_{\tau}] = \mathbb{E}[S_0]$.
But
$\mathbb{E}[S_0] = 0$, and therefore
$\mathbb{E}[S_{\tau}] = 0$,
which is  equivalent to \eqref{stopping}.
\end{proof}
\section*{Example \ref{ex:lip}  (continued)}
Let $B(J)$ denote the ball of radius $J$ centered at the origin in $\reals^q$. We show that
$$
\sup_{y\in B(J), x \in \reals^{d}}
\left\|\frac{\partial^2 f }{\partial x\partial y}(y,x)\right\|
\leq
\sup_{y \in B(J), \|x\|\leq\sqrt{d} }
\left\|\frac{\partial^2 g}{\partial x\partial y}(y,x)\right\|.
$$
Note that the right hand side is finite, as it is the supremum of a continuous function over a compact set. For ease of notation, let $A(u,v)$ denote the result of applying the bilinear map $A$ to the argument $(u,v)$.
For example, if
$\tfrac{\partial^2 f}{\partial x\partial y}(y,x)$
is the mixed-partial of $f$ at $(y,x)$, and
$(u,v) \in \reals^q \times \reals^d$,
then
$\tfrac{\partial^2 f}{\partial x\partial y}(y,x)(u,v)$
is the number
$\sum\limits_{i=1}^{q}
\sum\limits_{j=1}^{d}
\tfrac{\partial^2 f}{\partial x_j\partial y_i}(y,x)u_iv_j
$.
Using this notation, we have
\begin{equation}\label{starter-pack}
  \begin{split}
    \sup_{y \in B(J), x \in \reals^{d}}
    \left\|\frac{\partial^{2} f}{\partial x\partial y}(y,x)\right\|
    &=
    \sup_{y \in B(J), x \in \reals^{d}} \sup_{\|u\|=1,\|v\|=1}
    \left|\frac{\partial^{2} f}{\partial x\partial y}(y,x)(u,v)\right|\\
    &=
    \sup_{y \in B(J), x \in \reals^{d}} \sup_{\|u\|=1,\|v\|=1}
    \left|\frac{\partial^{2} g}{\partial x\partial y}(y,h(x))\left(u,\frac{\partial h}{\partial x}(x)v\right)\right|
  \end{split}
  \end{equation}
  Next, note that for any $x \in \reals$, we have $|\tanh(x)| \leq 1$  and $\tanh'(x) \leq 1$. Therefore  $\|h(x)\| \leq \sqrt{d}$, and
  $\|\frac{\partial h}{\partial x}(x)\| \leq 1$. Continuing from \eqref{starter-pack}, then,
   \begin{align*}
    \sup_{y \in B(J), x \in \reals^{d}}
    \left\|\frac{\partial^{2} f}{\partial x\partial y}(y,x)\right\|
    &\leq
        \sup_{y \in B(J), x \in \reals^{d}}\left\|\frac{\partial h}{\partial x}(x)\right\| 
    \left\|\frac{\partial^{2} g}{\partial x\partial y}(y,h(x))\right\|
    \\&\leq
        \sup_{y \in B(J), x \in \reals^{d}} 
    \left\|\frac{\partial^{2} g}{\partial x\partial y}(y,h(x))\right\|\\
    &\leq
    \sup_{y \in B(J), \|x\|\leq\sqrt{d}}
    \left\|\frac{\partial^{2} g}{\partial x\partial y}(y,x)\right\|. 
  \end{align*}
  
 \section{Analysis of Biased SGD}
 \section*{Proof of Proposition \ref{prop:abstract-bias}}
  \begin{proof}
    For convenience, define the random variables $\delta_t$ for $t=1,2,\hdots$ as
    $\delta_t = v_t - \nabla f_T(x_t)$.
    From \eqref{eqn:taylor}, it holds that
    \begin{align*}
    f_T(x_{t+1})
    &\leq f_T(x_t)
    - \eta\nabla f_T(x_t)^{T}\left( \nabla f_T(x_t) + \delta_t + \Delta_t\right) 
    + \frac{L}{2}\eta^2\|\nabla f_T(x_t) + \delta_t + \Delta_t\|^{2}.
    \end{align*}
    Summing this inequality over $t=1,\hdots,k$ for an arbitrary $k\geq 1$ yields
    \begin{equation}\label{eqn:biased-summer}
      \begin{split}
        f_T(x_{k+1})
    &\leq
    f_T(x_1)
    -
    \eta\sum\limits_{t=1}^{k}\nabla f_T(x_t)^{T}\left( \nabla f_T(x_t) +  \delta_t + \Delta_t\right) 
    +
    \frac{L}{2}\eta^2\sum\limits_{t=1}^{k}\|\nabla f_T(x_t) + \delta_t + \Delta_t\|^{2} \\
    &=
    f_T(x_1)
    -
    \eta\left(1-\frac{L}{2}\eta\right)\sum\limits_{t=1}^{k}\|\nabla f_T(x_t)\|^{2} 
    -
    \eta(1- L\eta)\sum\limits_{t=1}^{k}\nabla f_T(x_t)^{T}\delta_t 
    \\&\quad+
    \frac{L}{2}\eta^2\sum\limits_{t=1}^{k}\|\delta_t\|^{2} 
    -\eta(1-L\eta)\sum\limits_{t=1}^{k}\nabla f_T(x_t)^T\Delta_t
    +\frac{L}{2}\eta^2\sum\limits_{t=1}^{k}\|\Delta_t\|^{2}
    +L\eta^2 \sum\limits_{t=1}^{k} \delta_t^T\Delta_t.
      \end{split}
    \end{equation}
    In general, for any numbers $a,b$ it is the case that 
    $|ab| \leq \frac{1}{2}a^{2} + \frac{1}{2}b^2$.
    Then
    \begin{equation}\label{eqn:hence1}
      |\delta_{t}^{T}\Delta_t| \leq
      \|\delta_{t}\|\|\Delta_t\| \leq \frac{1}{2}\|\delta_t\|^{2} + \frac{1}{2}\|\Delta_t\|^2
      \end{equation}
    and 
    \begin{equation}\label{eqn:hence2}
      \begin{split}
        |\nabla f_T(x_t)^{T}\Delta_t| &\leq \|\nabla f_T(x_t)\|\|\Delta_t\|
        \leq \frac{1}{2}\|\nabla f_T(x_t)\|^2 + \frac{1}{2}\|\Delta_t\|^2.
      \end{split}
    \end{equation}
    Combining  \eqref{eqn:biased-summer}, \eqref{eqn:hence1}, \eqref{eqn:hence2}, and the fact that $\eta \leq 1/L$, we obtain
    \begin{equation*}
      \begin{split}
        f_T(x_{k+1}) &\leq
        f(x_1)
        -
        \frac{\eta}{2}
        \sum\limits_{t=1}^{k}
        \|\nabla f_T(x_t)\|^{2}
        -
    \eta(1- L\eta)\sum\limits_{t=1}^{k}
    \nabla f_T(x_t)^{T}\delta_t
    +
    L\eta^2\sum\limits_{t=1}^{k}
    \|\delta_t\|^{2} 
    +  \frac{\eta}{2}(1 + L\eta)\sum\limits_{t=1}^{k}\|\Delta_t\|^{2}.
      \end{split} 
      \end{equation*}
    \nobreak
    Rearranging terms, while noting that
    $f_T(x_{k+1}) \geq f^{*}$, $\|\Delta_t\|^2 \leq V_t,$ and $\eta\leq 1/L$, then,
    \begin{equation}
      \begin{split}
      \frac{\eta}{2}\sum\limits_{t=1}^{k}
      \|\nabla f_T(x_t)\|^{2}
    &\leq  
    f_T(x_1)
    -
    f^{*}
    -
    \eta(1- L\eta) \sum\limits_{t=1}^{k}
    \nabla f_T(x_t)^{T}\delta_t
    +
    L\eta^2\sum\limits_{t=1}^{k}
    \|\delta_t\|^{2}
    +
    \eta\sum\limits_{t=1}^{k}  V_t.  \label{bias-rearranged}
      \end{split}
      \end{equation}
    For each
    $n\geq 1$
    define
    $\tau(\epsilon) \wedge n$
    to be the stopping time which is the minimum of
    $\tau(\epsilon)$ and the constant $n$.
    Using Proposition \ref{mart-like-thm} with Assumption \ref{asu:vanilla-grad}, it holds that
    \begin{equation}\label{bias-zero}
    \mathbb{E}\left[
      \sum\limits_{t=1}^{\tau(\epsilon) \wedge n}\nabla f_T(x_t)^{T}\delta_t
      \right] = 0
    \end{equation}
    and
    \begin{equation}\label{bias-std}
    \mathbb{E}\left[
      \sum\limits_{t=1}^{\tau(\epsilon) \wedge n}\|\delta_t\|^{2}
      \right]
    \leq
    \sigma_v^2\mathbb{E}[\tau(\epsilon) \wedge n].
    \end{equation}
    Next, according to conditions \eqref{v1-eqn},  and \eqref{vn-eqn}, it holds that for any $k\geq 1$,
    \begin{equation}\label{vcons}
      \sum\limits_{t=1}^{k}V_t \leq \alpha\sum\limits_{t=1}^{k}V_t + \sum\limits_{t=1}^{k}U_t + \beta
    \end{equation}
    and by  \eqref{bn-eqn} together with Proposition \ref{mart-like-thm},
    \begin{equation}\label{vcons2}
      \mathbb{E}\left[\sum\limits_{t=1}^{\tau(\epsilon) \wedge n}U_t\right]
    \leq
    \beta\,\mathbb{E}[\tau(\epsilon)\wedge n].
    \end{equation}
    Combining \eqref{vcons} and \eqref{vcons2}, then 
    $$
    \mathbb{E}\left[
      \sum\limits_{t=1}^{\tau(\epsilon) \wedge n}
      V_t
      \right]
    \leq
    \alpha
    \mathbb{E}\left[
      \sum\limits_{t=1}^{\tau(\epsilon) \wedge n}
      V_t
      \right]
    +
    \beta\left(\mathbb{E}[\tau(\epsilon\wedge n)] + 1\right)
    $$
    which, upon rearranging, results in
    \begin{equation}\label{bias-results}
      \mathbb{E}\left[
        \sum\limits_{t=1}^{\tau(\epsilon)\wedge n}V_t
        \right]
      \leq
      \frac{\beta}{1-\alpha}\left( \mathbb{E}[\tau(\epsilon)\wedge n]+1\right).
    \end{equation}

    Combining
    \eqref{bias-rearranged},
    \eqref{bias-zero},
    \eqref{bias-std},
    \eqref{bias-results} and
    results in
    \begin{equation}\label{amazingly-elite}
      \begin{split}
        \frac{\eta}{2}\mathbb{E}\left[\sum\limits_{t=1}^{\tau(\epsilon) \wedge n}\|\nabla f_T(x_t)\|^2\right] \leq
    f_T(x_1) - f^{*}
    +
    L\eta^2\sigma_v^2\mathbb{E}\left[\tau\left(\epsilon\right) \wedge n\right]
    +
    \eta\frac{\beta}{1-\alpha}\left(\mathbb{E}[\tau(\epsilon) \wedge n]  + 1 \right).
      \end{split}
    \end{equation}
   Next, it follows from  squaring  \eqref{grad-dist} that for all $x \in \reals^d$,
    \begin{equation}\label{squared-wass-bound}
    \|\nabla f_{V}(x)\|^2 \leq 2 G^2d_1(\mu_V,\mu_T)^2 + 2\|\nabla f_{T}(x)\|^2.
    \end{equation}
    and for any $k\geq 1$,
    \begin{equation}\label{noter}
      \frac{k}{m} \leq \sum\limits_{t=1}^{k}1_{t\equiv 1 \Mod m}  \leq \frac{k}{m} + 1.
    \end{equation}
    Combining \eqref{squared-wass-bound} and \eqref{noter} results in
    \begin{equation}\label{mod-sum-ineq}
        \begin{split}
    \sum\limits_{t=1}^{\tau(\epsilon)\wedge n}
    1_{t\equiv 1 \Mod m}\|\nabla f_{V}(x_t)\|^2 
    &\leq 
    2\sum\limits_{t=1}^{\tau(\epsilon)\wedge n}
    1_{t\equiv 1 \Mod m} G^2 d_1(\mu_V,\mu_T)^2 
    + 
    2\sum\limits_{t=1}^{\tau(\epsilon)\wedge n}
    1_{t \equiv 1 \Mod m}\|\nabla f_{T}(x_t)\|^2  \\
    &
    \leq
    2 G^2 d_1(\mu_V,\mu_T)^2 \left(\frac{(\tau(\epsilon) \wedge n)}{m} + 1\right) +
    2 \sum\limits_{t=1}^{\tau(\epsilon)\wedge n}
    \|\nabla f_{T}(x_t)\|^2.
    \end{split}
   \end{equation}
    Furthermore, combining \eqref{noter} with the definition of $\tau$,
    \begin{equation}\label{bias-tau-prop}
      \begin{split}
     \mathbb{E}\left[ \sum\limits_{t=1}^{\tau(\epsilon) \wedge n }
     1_{t\equiv 1 \Mod m}\|\nabla f_V(x_t)\|^{2} \right]
      &\geq
     \mathbb{E}\left[ \sum\limits_{t=1}^{(\tau(\epsilon) \wedge n) -1}
     1_{t\equiv 1 \Mod m}\|\nabla f_V(x_t)\|^{2} \right]\\
      &\geq
     \frac{\epsilon}{m} \mathbb{E}[(\tau(\epsilon) \wedge n) -1].
      \end{split}
    \end{equation}
    Combining \eqref{amazingly-elite}, \eqref{mod-sum-ineq} and \eqref{bias-tau-prop},
    \begin{align*}
    \frac{\eta \epsilon}{4 m}
    \left( \mathbb{E}[\tau(\epsilon) \wedge n] - 1 \right)
    \leq 
    &\frac{\eta}{2} G^2 d_1(\mu_V,\mu_T)^2\left( \frac{\mathbb{E}[\tau(\epsilon) \wedge n]}{m} + 1\right) +    \\&
    f_T(x_1) - f^{*}
    +
    L\eta^2\sigma_v^2\mathbb{E}\left[\tau\left(\epsilon\right) \wedge n\right]
    +
    \frac{\eta\beta}{1-\alpha}\left(\mathbb{E}[\tau(\epsilon) \wedge n]  + 1 \right).
    \end{align*}
    This can be rearranged into
     \begin{equation*}
      \begin{split}
        \left(
        \frac{\eta \epsilon}{2m} - 2 L\eta^2\sigma_v^2 - \frac{ 2\eta\beta}{1-\alpha}
        - \frac{\eta}{ m }G^2 d_1(\mu_V,\mu_T)^2
        \right)
        \mathbb{E}[\tau(\epsilon) \wedge n] 
    &\leq 
    \eta G^2 d_1(\mu_V,\mu_T)^2 \\&+ 
    2 ( f_T(x_1) - f^{*} )
    +
\frac{2 \eta\beta }{1-\alpha} +
    \frac{\eta\epsilon}{2m},
      \end{split}
    \end{equation*}
    which in turn is equivalent to
    \begin{equation}\label{eqn:biased-prebound}
    \mathbb{E}[\tau(\epsilon) \wedge n]
    \leq
    \frac{\eta G^2 d_1(\mu_V,\mu_T)^2 + 2(f_T(x_1) - f^{*}) + \eta\epsilon/(2m)  + 2\eta\beta /(1-\alpha)}
         {\eta\epsilon/(2m) - 2 L\eta^2\sigma_v^2 -2 \eta\beta /(1-\alpha) - \eta G^2d_1(\mu_V,\mu_T)^2/m}.
    \end{equation}
    Note that the sequence of random variables
    $\{ (\tau (\epsilon) \wedge n)\}_{n=1,2,\hdots}$
    is monotone increasing, and converges pointwise to
    $\tau(\epsilon)$.
    Then the claimed bound on the expected time 
     follows from
    \eqref{eqn:biased-prebound} by the monotone convergence theorem.
    
    Using  \eqref{grad-dist}, we see that
    \begin{align*}
      \|\nabla f_{T}(x_{\tau(\epsilon)})\|
      \leq
      \|\nabla f_{V}(x_{\tau(\epsilon)})\| + G d_1(\mu_V,\mu_T).
    \end{align*}
    Using the definition of $\tau(\epsilon)$ and squaring each sides of this equation yields  \eqref{holder-nabla}.
  \end{proof}
  \section*{Proof of Corollary \ref{cor:bias-cvg}}
  
  \begin{proof}
    According to the definition of the step-size $\eta$ \eqref{step-form},
    \begin{equation}\label{denom-lb}
    \begin{split}
    \eta\left[ (\epsilon/2 - G^2 d_1(\mu_V,\mu_T)^2)/m - \eta(2 L\sigma_v^2 + 2 R  /(1-\alpha))\right]
    \geq
    \eta(1-c)(\epsilon/2 - G^2d_1(\mu_V,\mu_T)^2)/m
    \end{split}
    \end{equation}
    and
    \begin{equation}\label{one-over-eta}
    \frac{1}{\eta} 
    \leq 
    \frac{L}{c} + 
    \frac{m(2L\sigma_v^2 + 2 R /(1-\alpha))}
         {c\,( \epsilon/2 - G^2 d_1(\mu_V,\mu_T)^2)}.
    \end{equation}
    Combining these inequalities with  Proposition \ref{prop:abstract-bias}  yields
    \begingroup
    \addtolength{\jot}{1em}
    \begin{equation}\label{starter-eqn-abc}
    \begin{split}
    \mathbb{E}[\tau(\epsilon)]
    &\stackrel{\textbf{A}}{\leq}
    \frac{
      \eta G^2 d_1(\mu_V,\mu_T)^2 + 2(f_T(x_1) - f^{*}) + \eta\epsilon/(2m) + 2\eta \beta /(1-\alpha)
    }
    {\eta
      (\epsilon/2 - G^2 d_1(\mu_V,\mu_T)^2)/m - 2 L\eta^2\sigma_v^2 - 2\eta\beta /(1-\alpha)
    }.
    \\
    &\stackrel{\textbf{B}}{=}
    \frac{
      \eta G^2 d_1(\mu_V,\mu_T)^2 + 2(f_T(x_1) - f^{*}) + \eta\epsilon/(2m) + 2\eta^2 R /(1-\alpha)
    }
         {
           \eta(\epsilon/2 - G^2 d_1(\mu_V,\mu_T)^2)/m- 2 L\eta^2\sigma_v^2 - 2\eta^2 R /(1-\alpha)
         }.
    \\   
    &\stackrel{\textbf{C}}{\leq} 
    \frac{
      \eta G^2 d_1(\mu_V,\mu_T)^2  + 2(f_T(x_1) - f^{*}) + \eta\epsilon/(2m) + 2\eta^2 R /(1-\alpha)
    }
         {\eta(1-c)\,(\epsilon/2 - G^2 d_1(\mu_V,\mu_T)^2)/m}.
    \end{split}
    \end{equation}
    \endgroup
        Step \textbf{A} was established by Proposition \ref{prop:abstract-bias}, step \textbf{B} uses the assumption that $\beta = \eta R$, and step \textbf{C} is an application of \eqref{denom-lb}. Next, we will upper-bound the final inequality in three steps. First, using  \eqref{one-over-eta}, we see that
            \begingroup
    \addtolength{\jot}{1em}
    \begin{equation}\label{abc-1}
    \begin{split}
    &       \frac{ 2(f_T(x_1) - f^{*})}
         {\eta(1-c)\,(\epsilon/2 - G^2 d_1(\mu_V, \mu_T)^2)/m} 
    \leq 
    \\&\quad\quad\quad\quad
        \frac{2L m(f_T(x_1) - f^*)}{(1-c)\,c\, (\epsilon/2 - G^2 d_1(\mu_V,\mu_T)^2)}
        +
         \frac{2 m^2 (f_T(x_1) - f^{*})\left( 2L\sigma_v^2 + 2 R /(1-\alpha)\right)}{(1-c)\,c\,(\epsilon/2 - G^2 d_1(\mu_V,\mu_T)^2)^2}.
    \end{split}
    \end{equation}
  \endgroup
  Next, using the assumption on $\eta$, we have
      \begingroup
    \addtolength{\jot}{1em}
    \begin{equation}\label{abc-2}
        \begin{split}
        \frac{2 \eta^2 R /(1-\alpha)}
         {\eta (1-c)\,(\epsilon/2 - G^2 d_1(\mu_V,\mu_T)^2)/m} &=
          \frac{2 \eta R /(1-\alpha)}
         { (1-c)\,(\epsilon/2 - G^2 d_1(\mu_V,\mu_T)^2)/m} \\
         &
         \leq
           \frac{2  R /(1-\alpha)}
         {(1-c)\,(\epsilon/2 - G^2 d_1(\mu_V,\mu_T)^2)/m} \times \frac{ c\left( \epsilon/2 - G^2 d_1(\mu_V,\mu_T)^2\right)}{m ( 2 L \sigma_v^2 + 2 R / ( 1 - \alpha )) } \\
         &
         =
                \frac{c }
         {(1-c)}\times \frac{ 2 R /(1-\alpha)}{( 2 L \sigma_v^2 + 2 R / ( 1 - \alpha )) } 
         \leq
             \frac{c}
         {(1-c)}.
         \end{split}
    \end{equation}
    \endgroup
    Finally,
        \begingroup
    \addtolength{\jot}{1em}
    \begin{equation}\label{abc-3}
    \begin{split}
         \frac{\eta G^2 d_1(\mu_V,\mu_T)^2 + \eta\epsilon/(2m) }
         {\eta(1-c)\,(\epsilon/2 - G^2 d_1(\mu_V,\mu_T)^2)/m} &=
         \frac{G^2 d_1(\mu_V,\mu_T)^2 + \epsilon/(2m) }
         {(1-c)\,(\epsilon/2 - G^2 d_1(\mu_V,\mu_T)^2)/m}\\
         &
         =
         \frac{ m c G^2 d_1(\mu_V,\mu_T)^2 + c \epsilon/2}{(1-c)c(\epsilon/2 - G^2 d_1(\mu_V,\mu_T)^2)}.
    \end{split}
    \end{equation}
    \endgroup
    Above, the first step involved removing a common factor of $\eta$, and in the second step the result is multiplied by $(mc)/(mc)$. 
    Combining
    \eqref{starter-eqn-abc} with
    \eqref{abc-1},
    \eqref{abc-2}, and
    \eqref{abc-3}, we find that
        \begingroup
    \addtolength{\jot}{1em}
    \begin{equation}\label{bias-conc}
    \begin{split}
        \mathbb{E}[\tau(\epsilon)] &\leq
              \frac{4 m^2 (f_T(x_1) - f^{*})\left(L\sigma_v^2 +  R /(1-\alpha)\right)}{(1-c)\,c\,(\epsilon/2 - G^2 d_1(\mu_V,\mu_T)^2)^2}
              \\&+
           \frac{ 2 L m(f_T(x_1) - f^*) + m c G^2 d_1(\mu_V,\mu_T)^2 + c\epsilon/2}{(1-c)\,c\, (\epsilon/2 - G^2 d_1(\mu_V,\mu_T)^2)} + \frac{c}{1-c}.
    \end{split}
    \end{equation}
    \endgroup
  \end{proof}
  
  \section*{Proof of Corollary \ref{sgd-ifo}}
  
  \begin{proof}
  If the algorithm runs until iteration $\tau(\epsilon)$, then the number of times that the full gradient of $f_{V}$ is calculated is $\lceil\tau(\epsilon)/m\rceil \leq \tau(\epsilon)/m + 1$, and the number of IFO calls for the training function is $\tau(\epsilon)-1$. Therefore
 
    \begin{equation}\label{ifo-tau-ineq}
    \mathrm{IFO}(\epsilon) 
    \leq 
    \left(\frac{\tau(\epsilon)}{m} + 1 \right)n_V + (\tau(\epsilon)-1) 
    \leq 
    \tau(\epsilon)\left( \frac{n_V}{m} + 1\right) + n_V.
    \end{equation}
  Note that under our assumption on the gradient estimates $v_t$, we are in the unbiased setting where $R=0$. 
  Combining \eqref{bias-conc}  and \eqref{ifo-tau-ineq}, we obtain
  \begin{equation*}
      \begin{split}
        \mathbb{E}[\mathrm{IFO}(\epsilon)] 
        \leq 
            &\Bigg(\frac{4 m^2 (f_T(x_1) - f^{*})L\sigma_v^2 }{(1-c)\,c\,(\epsilon/2 - G^2 d_1(\mu_V,\mu_T)^2)^2}
            +
           \frac{2 Lm(f(x_1) - f^*) + m c G^2 d_1(\mu_V,\mu_T)^2 + c\epsilon/2}{(1-c)\,c\, (\epsilon/2 - G^2 d_1(\mu_V,\mu_T)^2)} + \frac{c}{1-c}\Bigg)
           \\&
           \times \left(\frac{n_V}{m} + 1\right) + n_V.
        \end{split}
  \end{equation*}
  Using $c=1/2$ and neglecting terms of lower order in $\epsilon$, then,
  \begin{equation}\label{neglector}
    \begin{split}
   \mathbb{E}[\mathrm{IFO}(\delta)]
   &= \mc{O}\left(
   \frac{
    m n_V + m^2
   }
   {(\epsilon -  2 G^2 d_1(\mu_V,\mu_T)^2)^2}
   + n_V\right) \\
      \end{split}
    \end{equation}
\end{proof}
  \section{Analysis of Decentralized SGD}
  The following result is a restatement of Lemma 5 of \citep{decentralized}.
  \begin{lem}\label{lemma-d1}
    Let Assumption \ref{asu:connectivity} hold. Then the limit $\lim_{k\to\infty}a^{k}$, which we denote $a^{\infty}$, is well defined and this matrix has entries
    $a^{\infty}_{i,j} = \frac{1}{M}$ for
    $1\leq i,j\leq M$.
    Furthermore, for all $k \geq 1$ it holds that $\|a^{\infty} - a^k\|^2 \leq \rho^k$.
  \end{lem}
  \section*{Proof of Proposition \ref{decentralized-dispersion}}
  \begin{proof}
    For $t\geq 1$ define
    $r_t$
    and $z_t$ to be the $(Md)$-dimensional vectors
    $r_{t} = \left( r_{t}^{1}, \hdots, r_{t}^{M} \right)$ and $z_{t} = \left( z_{t}^{1}, \hdots, z_{t}^{M} \right)$  respectively, 
  where, for $1\leq i \leq M$, the components $r_{t}^{i}, z_{t}^{i}$ are the $d$-dimensional vectors given by
  \begin{align}
    r_{t}^{i} &= x_{t}^{i} - \overline{x}_t, \label{yt-def}  \\
    z_{t}^{i} &= v_{t}^{i} - \frac{1}{M}\sum\limits_{j=1}^{M}v_{t}^j.    \label{zt-def}
  \end{align}
  Then we may express the variables $V_t$ as 
$$V_{t} = \frac{L^2}{M}\|r_t\|^2$$ 
  Let $a^{\infty}$ be the $M\times M$ matrix with entries
  $a^{\infty}_{i,j} = \frac{1}{M}.$ 
  Given matrices $A$ and $B$, we let $A \otimes B$ denote their Kronecker product.
  Then according to Line 5 of Algorithm \ref{algodgd}, the variables $r_{t}$ satisfy the recursion
  \begin{equation*}\label{y-eqn-decen-pr}
    r_{t+1} = \left((a - a^{\infty})\otimes I_{d}\right)r_t + \eta z_t.
  \end{equation*}
  Note that when $\|\cdot\|$ denotes the spectral norm on matrices, the Kronecker product satisfies $\|A\otimes B\| \leq \|A\|\|B\|$. Therefore, according to Lemma \ref{lemma-d1},
  \begin{equation}\label{y-eqn-decen}
    \|r_{t+1}\| \leq \sqrt{\rho}\|r_t\| + \eta \|z_t\|.
  \end{equation}
  Note that each $z^i_t$ can be expressed as
  \begin{equation}\begin{split}
      z^i_{t}   &= \nabla f(x_t^i) - \nabla f(\overline{x}_t)
      +
      v_t^i - \nabla f(x_t^i) -
      \frac{1}{M}\sum\limits_{j=1}^{M}(v_{t}^{j} - \nabla f(x_t^j))
      -
      \frac{1}{M}\sum\limits_{j=1}^{M}( \nabla f(x_t^j) - \nabla f(\overline x_t))
    \end{split}
  \end{equation}
Using the Lipschitz property of the gradient (Assumption \ref{asu:lipgrad}) then,
\begin{equation}\label{to-b-sqr}
 \|z^i_{t}\|
  \leq
  L\| x_t^i - \overline{x}_t \| 
  +
  \|v_t^i - \nabla f(x_t^i)\| +
  \frac{1}{M}\sum\limits_{j=1}^{M}
  \|v_{t}^{j} - \nabla f(x_t^j)\|
  +\frac{L}{M}\sum\limits_{j=1}^{M}\| x_t^j - \overline x_t \|.
  \end{equation}
Squaring and summing \eqref{to-b-sqr} over $i=1,\hdots,M$,
\begin{align*}
  \sum\limits_{i=1}^{M}\|z_t^i\|^2
  &\leq
  \sum\limits_{i=1}^{M}\left( L\| x_t^i - \overline{x}_t \| 
  +
  \|v_t^i - \nabla f(x_t^i)\| +
  \frac{1}{M}\sum\limits_{j=1}^{M}\|v_{t}^{j} - \nabla f(x_t^j)\|
  +\frac{L}{M}\sum\limits_{j=1}^{M}\| x_t^j - \overline x_t \|
  \right)^2 \\
    &\leq
  L^2 4\sum\limits_{i=1}^{M}\| x_t^i - \overline{x}_t \|^2 
  +
  4\sum\limits_{i=1}^{M}\|v_t^i - \nabla f(x_t^i)\|^2 \\&\quad
  +
  \frac{4}{M}\sum\limits_{i=1}^{M}\sum\limits_{j=1}^{M} \|v_{t}^{j} - \nabla f(x_t^j) \|^2
  +\frac{4 L^2}{M}\sum\limits_{i=1}^{M}\sum\limits_{j=1}^{M} \| x_t^j - \overline x_t \|^2 \\
  &=
  L^{2}8 \|r_t\|^2 +  8\sum\limits_{i=1}^{M} \|v_t^i - \nabla f(x_t^i)\|^2. 
\end{align*}
Taking square roots on each sides of this equation yields
\begin{equation}\label{cool-sqrt}
  \begin{split}
    \|z_t\|
    &\leq \sqrt{ L^{2}8 \|r_t\|^2 +  8\sum\limits_{i=1}^{M} \|v_t^i - \nabla f(x_t^i)\|^2 } 
    \leq  L\sqrt{ 8 } \|r_t\| +  \sqrt{8\sum\limits_{i=1}^{M} \|v_t^i - \nabla f(x_t^i)\|^2 }.
  \end{split}
  \end{equation}
Combining \eqref{y-eqn-decen} and \eqref{cool-sqrt}, then, 
\begin{align*}
\|r_{t+1}\|
&\leq \left(
\sqrt{\rho}
+
\eta L\sqrt{8}\right)\|r_t\| + \eta
\sqrt{8\sum\limits_{i=1}^{M} \|v_t^i - \nabla f(x_t^i)\|^2}.
\end{align*}
Squaring this equation, for any $k_1>0$ it holds that
\begin{equation}\label{multied}
\|r_{t+1}\|^2
\leq
(1+k_1)\left( \sqrt{\rho} + \eta L\sqrt{8}\right)^2\|r_t\|^2
+
8\eta^2 \left( 1 + \frac{1}{k_1} \right) \sum\limits_{i=1}^{M} \|v_t^i - \nabla f(x_t^i)\|^2.
\end{equation}
Define $k_1$ to be
$$ k_1 = \left( \frac{3+\sqrt{\rho}}{1+\sqrt{\rho}} \right)^{2}\frac{1}{4} - 1 $$
Then
$$
1 + \frac{1}{k_1}
=
\frac{9 + 6\sqrt{\rho} + \rho}{5 - 2\sqrt{\rho} - 3\rho}
\leq
\frac{16}{5 - 5\sqrt{\rho}}
\leq \frac{4}{1-\sqrt{\rho}}
$$
Multiplying each side of \eqref{multied} by $L^{2}/M$, 
it follows that
\begin{equation}\label{v-bd}
  \begin{split}
  V_{t+1}
  &\leq
  \left( \frac{3+\sqrt{\rho}}{1+\sqrt{\rho}} \right)^{2}
  \frac{1}{4}
  \left(\sqrt{\rho} + \eta L \sqrt{8}\right)^2 V_t
  +
  \frac{32\, \eta^2\, L^2}{M(1-\sqrt{\rho})}
  \sum\limits_{i=1}^{M}\|v_{t}^i - \nabla f(x_t^i)\|^2 \\
  &=
  \left( \frac{3+\sqrt{\rho}}{1+\sqrt{\rho}} \right)^{2}
  \frac{1}{4}
  \left(\sqrt{\rho} + \eta L \sqrt{8}\right)^2 V_t + U_t,
  \end{split}
\end{equation}
Using the assumption on $\eta$ \eqref{decentra-eta-cond}, it holds that
\begin{equation}\label{eta-holder}
  \begin{split}
    \sqrt{\rho} + \eta L \sqrt{8}
    &\leq \sqrt{\rho} + L \sqrt{8} \frac{1-\sqrt{\rho}}{4 L\sqrt{2}}
    = \frac{1+\sqrt{\rho}}{2}.
  \end{split}
  \end{equation}
Combining \eqref{v-bd} and \eqref{eta-holder}, then
\begin{align*}
  V_{t+1}
  &\leq
  \left(
  \frac{3+\sqrt{\rho}}
       {1+\sqrt{\rho}}
  \right)^2\frac{1}{4}
  \left(
  \frac{1+\sqrt{\rho}}
       {2}\right)^2
  V_t
  +
  U_{t} \\
  &=
  \frac{(3+\sqrt{\rho})^2}
       {16}
  V_t
  +
  U_{t}  \\
  &=
  \alpha V_t + U_{t},  
\end{align*}
It follows from the variance bound in Assumption \ref{asu:decentralized} that 
\begin{equation}\label{u-bd}
\mathbb{E}\left[U_{t} \mid \mc{F}_{t-1}\right]
\leq
\frac{32\,\eta^2\,L^2}{1-\sqrt{\rho}} \sigma^2_v.
\end{equation}
Combining \eqref{u-bd} with 
$
\eta
\leq
\frac{1-\sqrt{\rho}}
     {4 L\sqrt{2}}
     \leq
     \frac{1}{4L},
$
then
\begin{align*}
\mathbb{E}\left[ U_{t} \mid \mc{F}_{t-1} \right]
&\leq
\eta \frac{8   L \sigma_v^2}{1-\sqrt{\rho}}   
=
\beta.
\end{align*}
\end{proof}

  \section*{Proof of Proposition \ref{decentralized-main}}
  \begin{proof}
  To begin, note that the system average $\overline{x}_t$ satisfies the recursion
\begin{equation}\label{decentra-rec}
\overline{x}_{t+1} = \overline{x}_t + \frac{\eta}{M}\sum\limits_{i=1}^{M}v_{t}^{i}.
\end{equation}
Define the variables $v_t$ and $\Delta_t$, for $t\geq 1$, as
\begin{align*}
v_{t}
&=
\nabla f(\overline{x}_t)
+
\frac{1}{M}\sum\limits_{i=1}^{M}\left( v_t^i - \nabla f(x_t^i)\right) \\
\Delta_t &=
\frac{1}{M}\sum\limits_{i=1}^{M}
\left( \nabla f( x_t^i ) - \nabla f( \overline{x}_t ) \right)
\end{align*}
Then we can express the recursion \eqref{decentra-rec} as
$$ \overline{x}_{t+1} = \eta\left( v_{t} + \Delta_t\right) $$
We will show that this can be interpreted as a form of biased SGD and therefore we may apply Corollary \ref{cor:bias-cvg}.
For the unbiased component $v_t$, observe that
\begin{equation}\label{wow-unbiased}
  \mathbb{E}\left[ v_t - \nabla f_T(x_t) \mid \mc{F}_{t-1} \right]
  =
  \mathbb{E}\left[
    \frac{1}{M}
    \sum\limits_{i=1}^{M}
    (v_{t}^{i} - \nabla f(x_{t}^i))
\, \middle|\, \mc{F}_{t-1}  \right] = 0 
\end{equation}
and
\begin{equation}\label{wow-var}
  \mathbb{E}\left[ \|v_t - \nabla f_{T}(x_t)  \|^2 \mid \mc{F}_{t-1} \right] \leq
  \mathbb{E}\left[\frac{1}{M}\sum\limits_{i=1}^{M}\|v_{t}^{i} - \nabla f(x_{t}^i)\|^2\, \middle| \, \mc{F}_{t-1} \right] = \sigma_v^2
\end{equation}
For the bias term, note that
\begin{align*}
  \|\Delta_t\|^{2} \leq \frac{L^2}{M}\sum\limits_{i=1}^{M}\|x_{t}^i - \overline{x}_t\|^{2} = V_t
\end{align*}
Assumption \ref{asu:vanilla-grad}  follows from \eqref{wow-unbiased} and \eqref{wow-var}, while Assumption \ref{asu:bias-grad} follows from Proposition  \ref{decentralized-dispersion}. According to Corollary \ref{cor:bias-cvg}, then, a step-size of
\begin{equation}\label{step-dec}
\eta
=
c \cdot \min
\left\{
\frac{1}{L},
\frac{\epsilon/2 - G^2 d_{1}(\mu_v,\mu_T)^2}{m( 2L \sigma_v^2 + 2 R/(1-\alpha))}
  \right\}
\end{equation}
leads to
    \begingroup
    \addtolength{\jot}{1em}
    \begin{equation}\label{bias-conc-dec}
    \begin{split}
        \mathbb{E}[\tau(\epsilon)] &\leq
              \frac{4 m^2 (f_T(x_1) - f^{*})\left(L\sigma_v^2 +  R/(1-\alpha)\right)}{(1-c)\,c\,(\epsilon/2 - G^2 d_1(\mu_V,\mu_T)^2)^2}
              \\&+
              \frac{ 2 L m(f_T(x_1) - f^*) + m c G^2 d_1(\mu_V,\mu_T)^2 + c\epsilon}
                   {(1-c)\,c\, (\epsilon/2 - G^2 d_1(\mu_V,\mu_T)^2)}
                   +
                   \frac{c}{1-c}.
    \end{split}
    \end{equation}
    \endgroup
    In the present case,
    $R=8 L \sigma_v^2/(1-\sqrt{\rho})$
    and $1-\alpha = (7-6\sqrt{\rho} - \rho)/16$, so
    \begin{equation}\label{to-b-combined}
    \frac{R}{1-\alpha}
    =
    \frac{128 L \sigma_v^2}{(1-\sqrt{\rho})(7 - 6\sqrt{\rho} -\rho)}
    =
    \frac{128 L \sigma_v^2}{7 + 5\rho + \rho^{3/2} - 13\sqrt{\rho} }
    \end{equation}
    Combining \eqref{step-dec} with \eqref{to-b-combined} we arrive at  the definition of $\eta$ given in the statement of the proposition. 
    Furthermore,
    $$(1-\sqrt{\rho})(7 - 6\rho - \rho) \geq 7(1-\sqrt{\rho})(1-\sqrt{\rho})$$ so
    \begin{equation}\label{better-r-ineq}
      \frac{R}{1 - \alpha} \leq \frac{128 L \sigma_v^2}{7(1-\sqrt{\rho})^2}
    \end{equation}
    Combining \eqref{bias-conc-dec} with \eqref{better-r-ineq} we arrive at the claimed bound on $\mathbb{E}[\tau(\epsilon)]$.

    Finally, the condition  $c \leq \frac{1-\sqrt{\rho}}{4\sqrt{2}}$ is imposed to guarantee condition \eqref{decentra-eta-cond}.
  \end{proof}
  
  \section*{Proof of Corollary \ref{cor:decentralized-ifo}}
  \begin{proof}
    If DSGD runs until iteration
    $\tau(\epsilon)$,
    then number of times that the full gradient of $f_V$ is calculated is
    $\lceil \tau(\epsilon)/m\rceil \leq \tau(\epsilon)/m + 1$,
    and the number of IFO calls for the training function is $(\tau(\epsilon) - 1)M$.
    Therefore
    \begin{equation}\label{stacked-ifo-starter}
    \mathrm{IFO}(\epsilon) 
    \leq 
    \left( \frac{\tau(\epsilon)}{m} + 1\right)n_V + (\tau(\epsilon) -1)M
    \leq
    \tau(\epsilon)\left( \frac{n_V}{m} + M \right) + n_V.
    \end{equation}
    Next, note that
    $
    (1-c)c
    =
    (1-\sqrt{\rho})(4\sqrt{2} - 1 + \sqrt{\rho})/32 \geq (1-\sqrt{\rho})\sqrt{\rho}/32
    $,
    which implies
    \begin{equation}\label{c-eqn}
        \frac{1}{(1-c)c} \leq \frac{32}{(1-\sqrt{\rho})\sqrt{\rho}}.
    \end{equation}
    Combining \eqref{bias-conc-dec}, \eqref{stacked-ifo-starter}, and \eqref{c-eqn} we see that
    \begin{equation}\label{what-we-see-dsgd}
    \mathbb{E}\left[\mathrm{IFO}(\epsilon)\right] 
    =
    \mc{O}\left(
    \frac{m (n_V + m M)
    }{(1-\sqrt{\rho})^3\sqrt{\rho}(\epsilon - 2 G^2 d_1(\mu_V,\mu_T)^2)^2}+ n_V
    \right).\end{equation}

  \end{proof}

\section{Analysis of SVRG}

      For the analysis of SVRG, define the filtration $\{\mc{F}_{t}\}_{t=0,1,\hdots}$  as follows.
    $\mc{F}_{0} = \sigma( x_{m}^{1})$ and for all $s \geq 1$,
    $$\mc{F}_{s}
      =
      \sigma\left( \left\{ x_{m}^{1}\right\} \cup \left\{ y_{t}^{j} \, \middle|\, 0\leq t \leq m-1, \, 1 \leq j \leq s \right\} \right).$$
      
    We will leverage prior results concerning the behavior of SVRG. The following is adapted from  \citep{reddi2016stochastic}. It concerns conditions that guarantee expected descent of the objective function after each epoch.
  \begin{prop}\label{prop:svrg-prior}
     Let Assumptions \ref{asu:lipgrad} and \ref{asu:finite-sum-fns} hold. Let $\beta>0$ and
    define the constants $c_m,c_{m-1},\dots,c_0$ as follows:  $c_m=0$, and for $0 \leq t \leq m-1$, let
    $c_t = c_{t+1}(1+\eta \beta + 2 \eta^2 L^2) + \eta^2 L^3$.
    Define
    $\Gamma_t$ for $0\leq t \leq m-1$ as
    $\Gamma_t = \eta - \frac{c_{t+1} \eta}{\beta} - \eta^2 L - 2c_{t+1}\eta^2$.
    Suppose that the step-size
    $\eta$
    and the analysis constant $\beta$ are chosen so that
    $\Gamma_{t} >0 $ for $0\leq t\leq m-1$, and set $\gamma = \inf_{0\leq t < m} \Gamma_t$.
    Then for all $s\geq 1$,
    \begin{equation}\label{from-svrg-1}
      \begin{split}
        &\sum\limits_{t=0}^{m-1}\mathbb{E}[\|\nabla f_T(x_{t}^{s+1})\|^{2}\mid\mc{F}_{s-1}]
      \leq
      \frac{ 
        f_T(x_m^{s}) - \mathbb{E}[f_T(x_m^{s+1}) \mid \mc{F}_{s-1}]
      }{\gamma}.
      \end{split}
      \end{equation}
    Furthermore, if $\eta$ is of the form
    $\eta = \xi/(L n^{2/3})$
    for some $\xi \in (0,1)$ and if the epoch length is set to
    $m = \lfloor n/(3\xi) \rfloor,$
    then there is a value for $\beta$ such that
    \(
      \gamma \geq \frac{\nu(\xi)}{Ln^{2/3}} 
    \)
    where $\nu(\xi)$ is a constant dependent only on $\xi$. In particular, if $\xi = 1/4$  then
    \begin{equation}\label{from-svrg-3}
      \gamma \geq \frac{1}{ 40 L n^{2/3}}.
      \end{equation}
  \end{prop}
  \begin{proof}
    The proof of \eqref{from-svrg-1} follows from  nearly the same reasoning used to establish Equation (10) in (Section B, \citep{reddi2016stochastic}), the only difference being that conditional expectations replace expectations in all of the relevant formulas.

    Formula \eqref{from-svrg-3} follows from the proof of Theorem 3 given in (Appendix B, \citep{reddi2016stochastic}).
    \end{proof}
    \section*{Proof of Proposition \ref{prop:svrg-conv}}

  \begin{proof}
    First, note that $\tau(\epsilon)$ is a well-defined stopping time with respect to the filtration $\{\mc{F}_s\}_{s=0,1,\hdots}$.
    For $s=1,2,\hdots$
    define the random variables $\delta_{s}$ as
    \[
    \delta_{s} = \sum\limits_{t=0}^{m-1}\|\nabla f_T(x_{t}^{s+1})\|^{2} - \frac{f_T(x_m^{s}) - f_T(x_{m}^{s+1})}{\gamma}
    \]
    It holds trivially that for all $s\geq 1$,
    \begin{equation}\label{eqn:from-svrg}
      \sum\limits_{t=0}^{m-1}\|\nabla f_T(x_{t}^{s+1})\|^{2} = \frac{f_T(x_{m}^{s}) - f_T(x_m^{s+1})}{\gamma} + \delta_{s}
    \end{equation}
    and by  Proposition \ref{prop:svrg-prior} with $\xi = 1/4$,  for all $s\geq 1$,
    \begin{equation}\label{eqn:delta-zero}
      \begin{split}
        \mathbb{E}[\delta_{s} \mid \mc{F}_{s-1}]
        &=
        \sum\limits_{t=0}^{m-1}\mathbb{E}\left[\left\|\nabla f_T(x_{t}^{s+1})\right|^{2}\mid\mc{F}_{s-1}\right]
        -
        \frac{f_T(x_m^{s}) - \mathbb{E}[f_T(x_{m}^{s+1}) \mid \mc{F}_{s-1}] }{\gamma}\\
        &\leq 0. \end{split}
    \end{equation}
     Summing \eqref{eqn:from-svrg} over $s=1,\hdots,q$  yields 
    \begin{equation}\label{eqn:from-svrg-summed}
      \sum\limits_{s=1}^{q}\sum\limits_{i=0}^{m-1}\|\nabla f_T(x_{i}^{s+1})\|^{2}
      =
      \frac{f_T(x_m^{1}) -       f_T(x_{m}^{q+1})}{\gamma}
      + \sum\limits_{s=1}^{q}\delta_{s},
    \end{equation}
    Rearranging terms and noting that $f_T(x_m^{q+1}) \geq f^{*}$ results in
    \begin{equation}
    \label{eqn:from-svrg-rearranged}
      \gamma\sum\limits_{s=1}^{q}\sum\limits_{i=0}^{m-1}\|\nabla f_T(x_{i}^{s+1})\|^{2}
      \leq
      f_T(x_m^{1})
      - f^{*} + \gamma\sum\limits_{s=1}^{q}\delta_{s}.
    \end{equation}

  It follows that
  \begin{equation}\label{eqn:svrg-follow}
    \gamma \sum\limits_{s=1}^{q} \|\nabla f_T(x_{0}^{s+1})\|^{2}
      \leq
      f_T(x_m^{1})
      - f^{*} + \gamma\sum\limits_{s=1}^{q}\delta_{s}.
  \end{equation}
  For $r\geq 1$, let
  $\tau(\epsilon) \wedge r$ be the stopping time which is the minimum of
  $\tau(\epsilon)$ and the constant value $r$. 
  Applying Proposition \ref{mart-like-thm} together with \eqref{eqn:delta-zero}, it holds that
  \begin{equation}\label{eqn:svrg-delta-zero}
    \mathbb{E}\left[
    \sum\limits_{s=1}^{\tau(\epsilon)\wedge r}\delta_{s}
    \right] \leq 0
  \end{equation}
  Furthermore, by definition of $\tau$, 
  \begin{equation}\label{eqn:svrg-lb}
    \begin{split}
    \mathbb{E}\left[\sum\limits_{s=1}^{\tau(\epsilon) \wedge r}  \|\nabla f_T(x_{0}^{s+1})\|^{2}\right] 
     \geq
    \mathbb{E}\left[\sum\limits_{s=1}^{(\tau(\epsilon) \wedge r)-1} \|\nabla f_T(x_{0}^{s+1})\|^{2}\right]  
    \geq 
    &\mathbb{E}\left[\sum\limits_{s=1}^{(\tau(\epsilon) \wedge r)-1} \epsilon \, \right]  \\&\quad
    =
    \,\epsilon\,\mathbb{E}[(\tau(\epsilon) \wedge r)-1].
    \end{split}
  \end{equation}
  Combining \eqref{eqn:svrg-follow}, \eqref{eqn:svrg-delta-zero}, and \eqref{eqn:svrg-lb} yields
  $$
  \gamma\, \,\epsilon\, \mathbb{E}[(\tau(\epsilon) \wedge n) -1] \leq  f_T(x_{m}^{1}) - f^* 
  $$
  Rearranging terms in the above yields
  $$
  \mathbb{E}[ \tau(\epsilon) \wedge n] 
  \leq 
  \frac{f_T(x_{m}^1) - f^{*}}{\gamma\epsilon}  + 1.$$
  Applying the monotone convergence theorem, then,
  $$
  \mathbb{E}[ \tau(\epsilon)] 
  \leq
  \frac{f_T(x_{m}^1) - f^{*}}{\gamma \epsilon}  + 1.$$
  Next, specialize $\eta$ and $m$ to 
  $\eta = \xi/(Ln^{2/3})$  and $m = \lfloor n / (3\xi) \rfloor$ with $\xi=1/4$.
  Then by \eqref{from-svrg-3},
  $$\mathbb{E}[ \tau(\epsilon)] \leq \frac{40 L n^{2/3}(f_T(x_{m}^1) - f^{*})  }{ \epsilon}  + 1.$$
  
  \end{proof}

  \section{Generalization Analysis}
  \section*{Proof of Corollary \ref{cor:wass-svr}}
  \begin{proof}
    To begin, we establish that we may interchange derivatives and expectations in our definition of $f_G$, so that
    \begin{equation}\label{interchange}
      \nabla f_{G}(x) = \mathbb{E}_{y\sim \mu}\left[\nabla_x f(y,x)\right]
      \end{equation}
    To see why \eqref{interchange} holds, note first that under either of our Assumptions on $\mu$, the test distribution has a finite first moment:
    $\mathbb{E}_{y\sim\mu}[\|y\|] < \infty$.
    Then a sufficient  condition for
    \eqref{interchange} is that at each $x$ there be an Lipschitz function $g(y)$ such that
    $\|\nabla_{x} f(y,x+h)\| \leq g(y)$
    for all sufficiently small $h$  (Corollary 2.8.7 in \citep{bogachev2007measure}).
    Note that under Assumption \ref{asu:lipgrad}, it holds that,  $\|\nabla_{x} f(y,x+h)\| \leq \|\nabla_{x} f(y,x)\| + L\|h\|$.  Therefore, assume $\|h\|\leq 1$ and set $g(y) =\| \nabla_x f(y,x)\| + L$. Assumption \ref{asu:empirical-distance} guarantees that  $g$ is Lipschitz.

    Using \eqref{interchange} and following the reasoning used to establish \eqref{grad-dist}, it holds that
    \begin{align*}
        \|\nabla f_{G}(x_{\tau(\epsilon)}) -
        \nabla f_{T}(x_{\tau(\epsilon)})\|  &=
        \left\|
        \mathop{\mathbb{E}}\limits_{y\sim\mu}
        \left[\nabla_x f(y,x_{\tau(\epsilon)})\right] -
        \mathop{\mathbb{E}}_{y\sim\mu_T}\left[\nabla_x f(y,x_{\tau(\epsilon)})\right]
        \right\| \\
        &\leq
         G d_1(\mu,\mu_T).
    \end{align*}
    Therefore
    $$       \mathbb{E}[ \|\nabla f_{G}(x_{\tau(\epsilon)})\|] \leq 
        \mathbb{E}[\|\nabla f_{T}(x_{\tau(\epsilon)})\|  ]
        +
         G \mathbb{E}[d_1(\mu,\mu_T)].
$$
    Squaring and taking expectations, while noting that $d_1 \leq d_2$ (see  Remark 6.6 in \citep{villani2008}),
     \begin{align*}
       \mathbb{E}[ \|\nabla f_{G}(x_{\tau(\epsilon)})\|^2] \leq 
        2\mathbb{E}[\|\nabla f_{T}(x_{\tau(\epsilon)})\|^2  ]
        +
         2 G^2 \mathbb{E}[d_2(\mu,\mu_T)^2].
    \end{align*}
     If $J < \infty$, then we use the Wasserstein concentration bound from Theorem \ref{quantization} and the definition of $\tau(\epsilon)$ to obtain
     \begin{align*}
        \mathbb{E}[\|\nabla f_{G}(x_{\tau(\epsilon)})\|^2] \leq 
        2\epsilon  
        +
          2 G^2\kappa_d J n_V^{-3/d}.
     \end{align*}
     If $\mu$ is supported on at most $m$-points, then we may apply Theorem \ref{discrete-quantization}:
     \begin{align*}
        \mathbb{E}[\|\nabla f_{G}(x_{\tau(\epsilon)})\|^2] \leq 
        2\epsilon  
        +
          168 G^2\sqrt{\frac{m}{n_T}}.
     \end{align*}
  \end{proof}
  
\end{document}